\documentclass[twoside,leqno,12pt]{amsart} 
\usepackage{amsfonts} 
\usepackage{amsmath} 
\usepackage{amscd} 
\usepackage{amssymb} 
\usepackage{amsthm} 
\usepackage{url}
\usepackage{latexsym} 
\usepackage{color} 
\newtheorem{theorem}{Theorem} 
\newtheorem{lemma}{Lemma}

\newtheorem{corollary}{Corollary} 
 
\newtheorem{remark}{Remark}
\numberwithin{equation}{section}

\begin{document} 
\title{Diophantine approximation with prime denominator in quadratic number fields under GRH}

\author[S. Baier] {Stephan Baier}
\address{Stephan Baier\\ 
	Ramakrishna Mission Vivekananda Educational and Research Institute\\
	Department of Mathematics\\
	G.\ T.\ Road, PO~Belur Math, Howrah, West Bengal~711202\\
	India}
\email{stephanbaier2017@gmail.com}

\author[S. Das] {Sourav Das}
\address{Sourav Das\\
	Ramakrishna Mission Vivekananda Educational and Research Institute\\
	Department of Mathematics\\
	G.\ T.\ Road, PO~Belur Math, Howrah, West Bengal~711202\\
	India}
\email{sosourav007@gmail.com}

\author[E. A. Molla] {Esrafil Ali Molla}
\address{Esrafil Ali Molla\\
	Ramakrishna Mission Vivekananda Educational and Research Institute\\
	Department of Mathematics\\
	G.\ T.\ Road, PO~Belur Math, Howrah, West Bengal~711202\\
	India}
\email{esrafil.math@gmail.com}

\date{\today}



\begin{abstract}
Matom\"aki proved that if $\alpha\in \mathbb{R}$ is irrational, then there are infinitely many primes $p$ such that $|\alpha-a/p|\le p^{-4/3+\varepsilon}$ for a suitable integer $a$. In this paper, we extend this result to all quadratic number fields under the condition that the Grand Riemann Hypothesis holds for their Hecke $L$-functions.
\end{abstract}

\maketitle

\tableofcontents

\section{Introduction and statements of the main results}
A fundamental result in Diophantine approximation is the following theorem due to Dirichlet. 
\begin{theorem} \label{Dirich}
Given any real irrational $\alpha$, there are infinitely many pairs $(a,q)\in \mathbb{Z}\times \mathbb{N}$ of relatively prime integers such that 
$$ \left| \alpha - \frac{a}{q} \right| < q^{-2}.$$
\end{theorem}
While it is easy to prove the above statement just using the pigeonhole principle or the continued fraction expansion of $\alpha$, understanding Diophantine approximation by fractions with denominators restricted to subsets of the positive integers, such as the set of primes, is generally challenging. The problem of Diophantine approximation with primes has a rich history of research and triggered the development of several important tools in analytic number theory. In this context, the question becomes: For which $\theta > 0$ do there exist infinitely many prime numbers $p$ such that
\begin{equation*} 
\left|\alpha-\frac{b}{p}\right| < p^{-1-\theta+\varepsilon} 
\end{equation*} 
for a suitable $b\in \mathbb{Z}$? Equivalently, one may ask for which $\theta>0$ one can establish the infinitude of primes $p$ such that  
\begin{equation}\label{eq:DirichletApprox}
    ||p\alpha|| < p^{-\theta+\varepsilon},
\end{equation}
where $||x||$ is the distance of $x\in \mathbb{R}$ to the nearest integer.
In the following, we give a brief overview of the history of this problem.

In the thirties of the last century, I.\ M.\ Vinogradov \cite{vinogradov2004themethod} established the infinitude of primes $p$ satisfying \eqref{eq:DirichletApprox} with $\theta = 1/5$ by an intricate non-trivial treatment of trigometrical sums over primes, which also enabled him to prove the ternary Goldbach conjecture for sufficiently large odd integers.  Vaughan \cite{vaughan1978onthedistribution} improved the exponent to $\theta = 1/4$ by utilizing his celebrated identity for the von Mangoldt function and employing refined Fourier analytic arguments. Harman \cite{harman1983on-the-distribution} developed a new sieve method which enabled him to obtain the exponent $\theta=3/10$. Jia and Harman made several improvements, with Jia achieving the exponent $\theta=9/28$ in \cite{jia2000on-the-distribution} and Harman achieving $\theta=7/22$ in \cite{harman1996on-the-distribu}. Heath-Brown and Jia \cite{heath-brown2002the-distribution} introduced a new procedure to transform sieve functions into trigonometric sums and applied the Weil bound for Kloosterman sums to estimate certain trilinear sums, thus establishing the exponent  $16/49$, which is very close to $1/3$. Eventually, Matom\"aki \cite{matomaki2009the-distribution} unconditionally achieved the exponent $\theta=1/3$ by employing bounds for sums of Kloosterman sums. This is considered to be the limit of current technology. One might expect that $\theta=1-\varepsilon$ is an admissible exponent, but all known methods are far from establishing this, even conditionally under the Riemann Hypothesis. It was mentioned in \cite{heath-brown2002the-distribution} without proof that it is not difficult to obtain the exponent $\theta=1/3$, later established unconditionally by Matom\"aki, under the Riemann Hypothesis for Dirichlet $L$-functions.  This proof was carried out  in \cite{BaMo} by the first- and third-named authors. Throughout the sequel, we will abbreviate the Grand Riemann Hypothesis as GRH. 

Exploring analogues of this problem in number fields proves interesting, as novel ideas are needed to adapt classical methods to this context. The first-named author \cite{B2017} extended the classical problem in a slightly generalized form to $\mathbb{Q}(i)$, obtaining an exponent of  $\theta=1/12$. Introducing several novelties, Harman \cite{Harman2019} established the full analogue of his above-mentioned result with exponent $\theta=7/22$ for $\mathbb{Q}(i)$. Recently, utilizing Harman's sieve method, the first-named author, Mazumder and Technau established versions of these results for quadratic number fields of class number one in several papers (see \cite{BM}, \cite{BMT}, \cite{BT}). Ultimately, they achieved an analogue of Harman's exponent  $\theta=7/22$ in this specific context. In \cite{BaMo}, the first and third-named authors established an analogue of Matom\"aki's result ($\theta=1/3$) for the function field $k=\mathbb{F}_q(T)$ and its imaginary quadratic extensions by using the Riemann Hypothesis for Hecke $L$-functions of function fields, which is known to be true due to Weil's pioneering works. The arguments in \cite{BaMo} do not easily carry over to real quadratic extensions $K$ of $k$ without introducing new ideas. The main obstacle lies in the presence of an infinite group of units in the integral closure of $\mathbb{F}_q[T]$ in $K$ (see \cite{BaMo2} for a more detailed explanation). In \cite{BaMo2}, the authors used Vaughan's identity and exponential sums to obtain a result corresponding to Vaughan's exponent $\theta=1/4$ for real quadratic function fields.

In this article, we establish an analogue of Matom\"aki's result (exponent $\theta=1/3$) for all quadratic number fields under GRH for their Hecke $L$-functions. Our method is inspired by that in \cite{BaMo} in the case of imaginary quadratic fields. To handle the above-mentioned obstructions in the case of real quadratic fields, we use new ideas involving linear combinations of Hecke characters coming from the Fourier expansions of certain functions.  
Below are notations, used throughout this paper, followed by the statements of our main results for imaginary and real quadratic number fields. This section concludes with a brief description of our method.

\subsection{General notations}
\begin{itemize}
\item For a number field  $K$, we denote by $\mathcal{O}_K$ its ring of algebraic integers and by $\mathcal{O}_K^{\ast}$ the group of units in $\mathcal{O}_K$.
\item If $\mathfrak{a}$ is an ideal in $\mathcal{O}_K$, then we denote by $\mathcal{N}(\mathfrak{a})=[\mathcal{O}_K:\mathfrak{a}]$ its norm.
\item If $a$ is an element of $K$, then we denote by $\mathcal{N}(a)=|N_{K:\mathbb{Q}}(a)|$ the absolute value of its norm over $\mathbb{Q}$. We note that $\mathcal{N}(a)=\mathcal{N}(\mathfrak{a})$ if $\mathfrak{a}=(a)$ is a principal ideal in $\mathcal{O}_K$.    
\item Following usual convention, we will denote by $\varepsilon$ an arbitrarily small but fixed positive constant.
\item All $O$-constants in this paper are allowed to depend on $K$ and $\varepsilon$.
\end{itemize}

\subsection{Imaginary quadratic number fields} In this case, we shall establish the following. 

\begin{theorem}\label{firstmainresult}
Let $K$ be an imaginary quadratic number field and $\varepsilon$ be any positive real number. Assume that GRH holds for all Hecke $L$-functions of $K$. Then for any $\alpha\in \mathbb{C}\setminus K$, there exist infinitely many non-zero principal prime ideals $\mathfrak{p}$ in $\mathcal{O}_K$ such that 
$$
\left|\alpha - \frac{b}{p}\right|\le \mathcal{N}(\mathfrak{p})^{-2/3+\varepsilon}
$$ 
for some generator $p$ of $\mathfrak{p}$ and $b\in \mathcal{O}_K$.  
 \end{theorem}

The above Theorem \ref{firstmainresult} should be compared to the following version of the Dirichlet approximation theorem for imaginary quadratic number fields, which we will use in our proof of Theorem \ref{firstmainresult}. 

\begin{theorem} \label{Dirichlet2} 
Let $K$ be an imaginary quadratic number field. Then there exists a positive constant $C$ depending at most on $K$ such that for every $\alpha\in \mathbb{C}\setminus K$, there exist infinitely many elements $a/q\in K$ with $(a,q)\in\mathcal{O}_K\times(\mathcal{O}_K\setminus\{0 \})$ such that 
\begin{equation*}
    \left|\alpha-\frac{a}{q}\right| \leq \frac{C}{\mathcal{N}(q)}.
\end{equation*}
\end{theorem} 

\begin{proof}
This was established in \cite[section 2]{Gintner1936thesis} with a concrete constant $C$.
\end{proof}

\subsection{Real quadratic number fields} Now let $K=\mathbb{Q}(\sqrt{d})$ be a real quadratic number field, where $d>1$ is a square-free integer. In this context we use the following notations.
\begin{itemize}
\item We denote the two embeddings of $K$, the identity and conjugation, by 
$$ 
\sigma_1(s+t\sqrt{d}):=s+t\sqrt{d} 
$$ 
and 
$$ 
\sigma_2(s+t\sqrt{d}):=s-t\sqrt{d}, 
$$ 
where $s,t\in \mathbb{Q}$. 
\item We write 
\begin{align*} 
\sigma(K):=\{(\sigma_1(x),\sigma_2(x)) : x\in K\}. 
\end{align*} 
\end{itemize}

We shall establish the following. 

\begin{theorem} \label{secondmainresult}
Let $K$ be a real quadratic number field and $\varepsilon$ be any positive real number. Assume that GRH holds for all Hecke $L$-functions of $K$. Then for any pair $(x_1, x_2) \in \mathbb{R}^2\setminus \sigma(K)$, there exist infinitely many non-zero principal prime ideals $\mathfrak{p}$ in $\mathcal{O}_K$ such that 
$$ 
\left| x_i-\frac{\sigma_i(b)}{\sigma_i(p)}\right| \le \mathcal{N}(\mathfrak{p})^{-2/3 +\varepsilon} \quad \mbox{for } i=1,2 
$$ 
for some generator $p$ of $\mathfrak{p}$ and $b\in \mathcal{O}_K$. 
\end{theorem} 

The above Theorem \ref{secondmainresult} should be compared to the following version of the Dirichlet approximation theorem for real quadratic number fields, which we will use in our proof of Theorem \ref{secondmainresult}. 

\begin{theorem} \label{Dirichlet1} 
Let $K$ be a real quadratic number field. Then there exists a positive constant $C$ depending at most on $K$ such that for every pair $(x_1, x_2) \in \mathbb{R}^2 \setminus \sigma(K)$, there exist infinitely many elements $a/q\in K$ with $(a,q)\in\mathcal{O}_K\times(\mathcal{O}_K\setminus\{0 \})$ such that  
\begin{equation*} 
\left|x_i-\frac{\sigma_i(a)}{\sigma_i(q)}\right| \le \frac{C}{\mathcal{N}(q)} \quad \mbox{ for } i=1,2. 
\end{equation*} 
\end{theorem} 

\begin{proof}
This follows from a general version of the Dirichlet approximation theorem for number fields in  [\cite{Que}, Theorem 1].
\end{proof}

\begin{remark} {\rm For comparison, the unconditional main results in \cite{BMT} together with the considerations in \cite[section 12]{BaMo2} imply an exponent of $-\nu+\varepsilon$  with 
\begin{equation}\label{otherexponent}
\nu=1/2+7/44=\frac{29}{44}=0.65\overline{90}
\end{equation} 
in place of $-2/3+\varepsilon$
in both Theorems \ref{firstmainresult} and \ref{secondmainresult}, which is a weaker result since $\nu<2/3$.
In fact, in the said section 12 of \cite{BaMo2} it was pointed out that arguments from this paper can be incorporated in the treatment in \cite{BM} to obtain an exponent of $-\nu+\varepsilon$ with $\nu=1/2+1/8=5/8=0.625$ for the real quadratic case, which is an even weaker result. However, these arguments can also be incorporated in the treatment in \cite{BMT} to obtain an exponent of $-\nu+\varepsilon$ with $\nu$ as in \eqref{otherexponent} in this case.} \end{remark}
 
\begin{remark} {\rm It is to be noted that the above Theorems \ref{Dirichlet2} and \ref{Dirichlet1} do not assume $a$ and $q$ to be coprime (in the sense that the ideals generated by $a$ and $q$ are coprime). Though, if $K$ has class number one, then unique factorisation allows us to cancel any potential non-trivial common factors and assume $a$ and $q$ to be coprime. If the class number is greater than 1, then it may happen that $a$ and $q$ don't share common irreducible factors, but the ideals generated by them share non-trivial ideal factors (see the considerations at the beginning of subsection \ref{arb} below).} 
\end{remark}

\subsection{Description of our method} To establish our main results in Theorems \ref{firstmainresult} and \ref{secondmainresult} for imaginary and real quadratic number fields, we first approximate $\alpha$ as well as $(x_1,x_2)$ using Theorems \ref{Dirichlet2} and \ref{Dirichlet1}, respectively. This is followed by a reduction of our problem to counting prime elements in unions of certain residue classes. Their membership of residue classes is then detected using Dirichlet characters. The main contribution comes from the principal Dirichlet character to the relevant modulus. To evaluate this main term, we count prime elements with norm restricted by a suitable quantity. To bound the error term, we need to estimate sums of Dirichlet characters of prime elements. We relate them to sums of Hecke characters of principal prime ideals. This is rather easy for imaginary quadratic fields but requires novel ideas for real quadratic fields. In the latter case, we approximate the said sums by linear combinations of sums of Hecke characters with coefficients arising from the Fourier expansion of a certain function. In both the cases of imaginary and real quadratic fields, we first work out the easier case of class number one and then describe the required modifications needed to cover general class numbers.
Theorems \ref{firstmainresult} and \ref{secondmainresult} should be extendable to prime ideals in a given class of the class group, but for simplicity, we have confined ourselves to considering prime ideals in the class of principal ideals. For this more general setup, one may fix a class $C$ and an integral ideal $\mathfrak{a}_0$ in the class $C^{-1}$. Now one may consider Diophantine approximation with denominators generating a principal ideal of the form $\mathfrak{p}\mathfrak{a}_0$, where $\mathfrak{p}$ is a prime ideal in $C$. 

It is likely that our method can be extended to higher degree number fields.  The general question is, for some parameter $N$ tending to $+\infty$, how well elements $\alpha$ of $\mathcal{O}_K$ which are not in the image of the Minkowski embedding $j : K \hookrightarrow K_{\mathbb{R}} = K\otimes_{\mathbb{Q}} \mathbb{R}$ can be approximated using elements $j(b/p)$ ($b,p\in \mathcal{O}_K$, $p\not= 0$) subject to the principal ideal $p\mathcal{O}_K$ being prime and having its (absolute) norm $\asymp N$. This may be addressed in future research. However, in this article we decided to confine ourselves to quadratic fields since the computations in this case can be done relatively easily using a concrete description of Hecke characters. Here we first deal with the easier case of imaginary quadratic fields and then extend our methods to real quadratic fields. Extra difficulties when dealing with real quadratic fields come from the fact that for an integral ideal $\mathfrak{m}$, the subgroup of $(\mathcal{O}_K/\mathfrak{m})^{\times}$ consisting of cosets $u+\mathfrak{m}$ with $u$ unit in $\mathcal{O}_K$ can become very large (in fact, it can be whole of $(\mathcal{O}_K/\mathfrak{m})^{\times}$). For this reason, there are generally not enough Hecke characters modulo $\mathfrak{m}$ of trivial infinity type to detect residue classes, and therefore one needs to use Hecke characters of non-trivial infinity type as well. This is fundamentally different in the case of imaginary quadratic number fields, where the unit group is finite. \\ \\
{\bf Acknowledgements.} The authors are grateful to the referee for valuable comments. They would also like to thank the Ramakrishna Mission Vivekananda Educational and Research Insititute for an excellent work environment. The research of the second-named author was supported by a CSIR-NET Ph.D fellowship under file number 09/934(0012)/2019-EMR-l. The research of the third-named author was supported by a UGC NET grant under number NOV2017-424450. \\ \\
{\bf Competing interests statement:} There are no relevant financial or non-financial competing interests to report. 

\section{Hecke characters}
We will use Hecke characters to detect prime elements in segments of residue classes. These characters as well as their associated $L$-functions were introduced by Hecke (\cite{Hecke1918} and \cite{Hecke1920}) to obtain more refined information on the distribution of prime ideals in number fields than this is possible by just using the Dedekind zeta function. In a sense, Hecke characters generalize Dirichlet characters, which serve as a tool to detect residue classes in $\mathbb{Z}$. In addition to information about residue classes, Hecke characters also encode information about regions. To give an example, they can be used to count Gaussian primes in segments of sectors. In this paper, however, we will focus on prime elements in residue classes rather than regions. Hecke characters are defined on the group of non-zero fractional ideals of a number field $K$. This is in contrast to Dirichlet characters, which are defined on the elements of $\mathcal{O}_K$, the ring of integers of $K$. It is required to turn from elements of $\mathcal{O}_K$ to ideals in order to define associated $L$-functions in a meaningful way.
Below we give a fairly explicit description of Hecke characters, based on  material in \cite{Hecke1918}, \cite{Hecke1920}, \cite[section 3.3]{Miya} and \cite{mult}. 

\subsection{Notations}
\begin{itemize}
\item Let $x,y\in K$ and $\mathfrak{f}$ be an ideal in $\mathcal{O}_K$. Then $x$ and $y$ are said to be multiplicatively congruent modulo $\mathfrak{f}$, denoted as $x\equiv^{\ast} y \bmod{\mathfrak{f}}$, if the $\mathfrak{p}$-adic valuation of $(x-y)$ at any prime ideal $\mathfrak{p}$ dividing $\mathfrak{f}$ satisfies $v_{\mathfrak{p}}((x-y))\ge v_{\mathfrak{p}}(\mathfrak{f})$.  
\item If $\mathfrak{a}$ and $\mathfrak{b}$ are ideals in $\mathcal{O}_K$, then $\mathfrak{c}=$gcd$(\mathfrak{a},\mathfrak{b})$ is the unique integral ideal dividing both $\mathfrak{a}$ and $\mathfrak{b}$ such that if $\mathfrak{d}$ divides $\mathfrak{a}$ and $\mathfrak{b}$, then $\mathfrak{d}$ divides $\mathfrak{c}$. 
\item If $\mathcal{O}_K$ is a unique factorization domain and hence a gcd domain, then we write gcd$(a,b)\approx c$ if $c\in \mathcal{O}_K$ is a greatest common divisor of $a,b\in \mathcal{O}_K$. We note that if $K$ has class number one, then $\mathcal{O}_K$ is a gcd domain, and gcd$(a,b)\approx c$ is equivalent to gcd$((a),(b))=(c)$. 
\item If $\mathfrak{a}$ is a fractional and $\mathfrak{b}$ an integral ideal in $K$, then we say that $\mathfrak{a}$ is coprime to $\mathfrak{b}$ if $v_{\mathfrak{p}}(\mathfrak{a})=0$ for all prime ideals $\mathfrak{p}$ dividing $\mathfrak{b}$. 
\item By $\mathbb{P}$, we denote the set of non-zero prime ideals in $\mathcal{O}_K$. 
\end{itemize}

\subsection{Hecke characters on principal fractional ideals}
Hecke characters are certain characters of the multiplicative group of non-zero fractional ideals of a number field $K$ which are coprime to an integral ideal $\mathfrak{f}$, the modulus of the Hecke character. On the principal fractional ideals, they take the form
\begin{equation} \label{Heckechar}
\chi_{\text{Hecke}}((\alpha))=\chi(\alpha)\chi_\infty(\alpha) \, \text{ for all } \alpha\in K^{\times}, 
\end{equation}
with $\chi(\alpha)$ and $\chi_{\infty}(\alpha)$ having the properties described below.

We assume that $\chi$ is the multiplicative extension of a Dirichlet character $\tilde{\chi} \bmod \mathfrak{f}$ to $K$, where $\mathfrak{f} \subseteq \mathcal{O}_K$ is an integral ideal. (Here we mean by ''Dirichlet character" that $\tilde{\chi}$ is the pull-back of a group character for $(\mathcal{O}_K/\mathfrak{f})^{\ast}$ to 
$\mathcal{O}_K$.) To be precise, for $\alpha\in K$, this extension is defined as 
$$
\chi(\alpha):=\tilde{\chi}(a) \quad \mbox{if } a\in \mathcal{O}_K \mbox{ and } \alpha \equiv^{\ast} a \bmod{\mathfrak{f}}.   
$$
We note that $\chi$ is well-defined on $K$ in this way and refer to $\chi$ also as a Dirichlet character modulo $\mathfrak{f}$. It is not multiplicative on $K^{\times}$ but on the multiplicative subgroup 
$$
K(\mathfrak{f})=\left\{\alpha \in K^{\times} : v_{\mathfrak{p}}(\alpha)=0 \mbox{ for all prime ideals } \mathfrak{p} \mbox{ dividing } \mathfrak{f}\right\}.
$$ 

Let $\sigma_1,\dots, \sigma_{r_1}$ be the real and $\sigma_{r_1+1},\overline{\sigma_{r_1+1}},\dots,\sigma_{r_1+r_2},\overline{\sigma_{r_1+r_2}}$ be the complex embeddings of $K$, where $r_1+2r_2=n=[K:\mathbb{Q}]$. Denote the unit circle in the complex plane as $\mathbb{T}$. 
The infinity part $\chi_\infty$ on the right-hand side of \eqref{Heckechar} is supposed to have the property that there is a continuous character 
$$
\tilde{\chi}_{\infty} : (\mathbb{R}^{\times})^{r_1}\times (\mathbb{C}^{\times})^{r_2}\rightarrow \mathbb{T}
$$
such that 
$$
\chi_\infty(\alpha)=\tilde{\chi}_{\infty}(\sigma_1(\alpha),...,\sigma_{r_1+r_2}(\alpha)) \quad \mbox{for all } \alpha\in K(\mathfrak{f}).
$$
Explicitly, the infinity part looks as follows. (For a reference, see \cite[section 3.3]{Miya}.) For $\alpha\in K^\times$ let $\alpha_j=\sigma_j(\alpha)$. Then 
\begin{equation} \label{ChiInfty}
\chi_\infty(\alpha)= \displaystyle \prod_{j=1}^{r_1+r_2} \left((\alpha_j/|\alpha_j|)^{u_j}|\alpha_j|^{iv_j}\right),
\end{equation}
where $u_j, v_j$ are real numbers which satisfy the following conditions
\begin{enumerate}
\item $u_j\in \begin{cases}
\{0,1\} & \text{ for } j=1, ..., r_1, \\
\mathbb{Z} & \text{ for } j=r_1+1, ..., r_1+r_2,
\end{cases}$
\item $\sum\limits_{j=1}^{r_1+r_2} v_j=0$.
\end{enumerate}
The second condition above is not needed if we only demand $\chi_{\infty}$ to come from a continuous character $\tilde{\chi}_{\infty}$ but can be justified as follows. If 
$$
\frac{1}{r_1+r_2}\sum\limits_{j=1}^{r_1+r_2} v_j=\mu\not=0,
$$ 
then set $\tilde{v}_j:=v_j-\mu$. It follows that 
$$
\sum\limits_{j=1}^{r_1+r_2} \tilde{v}_j=0
$$ 
and 
$$
\prod\limits_{j=1}^{r_1+r_2}|\alpha_j|^{iv_j}=\mathcal{N}(\alpha)^{i\mu}\prod\limits_{j=1}^{r_1+r_2}|\alpha_j|^{i\tilde{v}_j}.
$$ 
Hence, condition (2) above can be forced at the expense of an additional factor of $\mathcal{N}(\alpha)^{i\mu}$, which amounts just to a shift of the corresponding $L$-function by $i\mu$ in vertical direction. 

If $\chi$ is given, then $\chi_\infty$ needs to be defined in such a way that $\chi(u)\chi_\infty(u)=1$ for all units $u\in\mathcal{O}_K^{\ast}$ so that the Hecke character $\chi_{\text{Hecke}}((\alpha))=\chi(\alpha)\chi_\infty(\alpha)$ is well-defined on the principal ideals $(\alpha)$. This restricts the choices of the exponents $u_j$ and $v_j$. In the next subsections, we work out these restrictions for real and imaginary quadratic field extensions, thus obtaining explicit expressions for all Hecke characters on the principal ideals in these cases. Throughout the sequel,
we say that a Hecke character $\chi_{\text{Hecke}}$ as given in \eqref{Heckechar} {\it belongs} to the Dirichlet character $\chi$.

\subsection{The case of imaginary quadratic extensions}
Let $K$ be an imaginary quadratic number field. In this case, \eqref{ChiInfty} takes the form
$$
\chi_{\infty}(\alpha)=\left(\frac{\alpha}{|\alpha|}\right)^u
$$
for a suitable $u\in \mathbb{Z}$. Let $\xi$ be a generator and $w$ be the order of the unit group $\mathcal{O}_K^{\ast}$. We have 
$$
w=\begin{cases} 6 & \mbox{ if } K=\mathbb{Q}(\sqrt{-3}),\\
4  & \mbox{ if } K=\mathbb{Q}(i), \\
2 & \mbox{ otherwise.}
\end{cases}
$$
It suffices to ensure that $\chi\chi_{\infty}(\xi)=1$. 

We note that $\chi(\xi)$ is an integral power of $\xi$ since $1=\chi(1)=\chi(\xi^w)=\chi(\xi)^w$ and $\xi$ is a primitive $w$-th root of unity. Suppose that 
$$
\chi(\xi)=\xi^{-k}.
$$
Then $\chi\chi_{\infty}(\xi)=1$ if and only if $u\equiv k \bmod{w}$, i.e. $\chi_{\infty}(\alpha)$ is of the form
$$
\chi_{\infty}(\alpha)=\left(\frac{\alpha}{|\alpha|}\right)^{wn+k} \quad \mbox{for some } n\in \mathbb{Z}.
$$

\subsection{The case of real quadratic extensions}
Now let $K$ be a real quadratic number field.  In this case, \eqref{ChiInfty} takes the form
\begin{equation}
\chi_\infty(\alpha)=\left|\frac{\sigma_1(\alpha)}{\sigma_2(\alpha)}\right|^{iv} \mbox{sgn}\left(\sigma_1(\alpha)\right)^{u_1} \mbox{sgn}\left(\sigma_2(\alpha)\right)^{u_2}
\end{equation}
for suitable real numbers $v$ and $u_1,u_2\in \{0,1\}$. The unit group $\mathcal{O}_K^{\ast}$ is generated by the fundamental unit $\epsilon>1$ and $-1$, and therefore it suffices to ensure that $\chi\chi_{\infty}(\epsilon)=1$ and $\chi\chi_{\infty}(-1)=1$.

Suppose that 
\begin{equation} \label{wdef}
\chi(\epsilon)\mbox{sgn}\left(\sigma_2(\epsilon)\right)^{u_2}= \exp(-2\pi i\gamma), \,\, \text{ i.e. } \,\, \gamma=\frac{l}{2}-\frac{\arg(\chi(\epsilon))}{2\pi},
\end{equation}
where 
$$
l:=\begin{cases} 0 & \mbox{ if } \mbox{sgn}\left(\sigma_2(\epsilon)\right)^{u_2}=1,\\
1 & \mbox{ if } \mbox{sgn}\left(\sigma_2(\epsilon)\right)^{u_2}=-1. \end{cases}
$$
Note that $\sigma_1(\epsilon)=\epsilon>0$ and hence 
$\mbox{sgn}\left(\sigma_1(\epsilon)\right)^{u_1}=1$.
To ensure that $\chi\chi_{\infty}(\epsilon)=1$, it therefore suffices to demand that 
$$\mbox{sgn}(\sigma_2(\epsilon))^{-u_2}\chi_\infty(\epsilon)=\left|\frac{\sigma_1(\epsilon)}{\sigma_2(\epsilon)}\right|^{iv}= \exp(2\pi i\gamma),
$$
i.e.,
\begin{equation*}
\epsilon^{2iv}=\exp(2iv\log \epsilon)=\exp(2\pi i\gamma). 
\end{equation*}
This is equivalent to  
\begin{equation*}
v= \frac{\pi(n+\gamma)}{\log \epsilon} \text{ for some } n\in \mathbb{Z}.
\end{equation*}
Moreover, if $\chi(-1)=1$, then we take $u_1=u_2$, and if $\chi(-1)=-1$, then we take $u_1\neq u_2$ to ensure that $\chi\chi_{\infty}(-1)=1$. So, altogether, the infinity part is of the form
\begin{equation} \label{infinitypart}
\chi_\infty(\alpha)=\left|\frac{\sigma_1(\alpha)}{\sigma_2(\alpha)}\right|^{i\pi(n+\gamma)/\log \epsilon}\cdot\mbox{sgn}(\sigma_1(\alpha))^{u_1}\cdot\mbox{sgn}(\sigma_2(\alpha))^{u_2}
\end{equation}
with $u_1,u_2\in\{0,1\}$ satisfying the above relations.

\subsection{Extension to general fractional ideals} Let $\mathfrak{f}$ be an integral ideal in $\mathcal{O}_K$. Let $I$ be the group of all non-zero fractional ideals, $P$ be the subgroup of non-zero principal fractional ideals, $I(\mathfrak{f})$ be the group of all non-zero fractional ideals coprime to $\mathfrak{f}$ and $P(\mathfrak{f})$ be the subgroup of non-zero principal fractional ideals coprime to $\mathfrak{f}$. It is known that
$$
\sharp\left(I(\mathfrak{f})/P(\mathfrak{f})\right)=\sharp\left(I/P\right)=h_K,
$$
the class number of $K$. In general, a Hecke character modulo $\mathfrak{f}$ is a function defined on $I$ which equals zero on all ideals in $I\setminus I(\mathfrak{f})$, whose restriction to $I(\mathfrak{f})$ is a character of $I(\mathfrak{f})$ and whose restriction to $P(\mathfrak{f})$ is a character as described in the previous subsections. Now we describe how a given Hecke character $\chi_{\text{Hecke}}$ on $P(\mathfrak{f})$ can be extended to Hecke characters on $I(\mathfrak{f})$. 

Since $I(\mathfrak{f})/P(\mathfrak{f})$ is an abelian group, we have 
$$
I(\mathfrak{f})/P(\mathfrak{f}) \cong \mathbb{Z}_{m_1} \oplus \cdots \oplus \mathbb{Z}_{m_k}
$$    
for suitable $m_1,...,m_k\in \mathbb{N}$. Let $i$ be the corresponding isomorphism. Then there are ideals $\mathfrak{a}_1,...,\mathfrak{a}_k$ in $I(\mathfrak{f})$ such that every element in $I(\mathfrak{f})/P(\mathfrak{f})$ can be uniquely written in the form $\mathfrak{a}_1^{j_1}\cdots \mathfrak{a}_k^{j_k}P(\mathfrak{f})$ with $0\le j_l\le m_l-1$ for $l=1,...,k$, where 
$$
i(\mathfrak{a}_lP(\mathfrak{f}))=(0,...,0,1,0,...,0)
$$  
with $1$ in $l$-th position. The set 
$$
\{\mathfrak{a}_1^{j_1}\cdots \mathfrak{a}_k^{j_k}: 0\le j_l\le m_l-1 \quad \mbox{for } l=1,...,k\}
$$
forms a system of coset representatives of $I(\mathfrak{f})/P(\mathfrak{f})$, where
$$
i(\mathfrak{a}_1^{j_1}\cdots \mathfrak{a}_k^{j_k}P(\mathfrak{f}))=(j_1,...,j_k).
$$
The ideals $\mathfrak{a}_l^{m_l}$ for $l=1,...,k$ are principal. For $\chi_{\text{Hecke}}$, given on $P(\mathfrak{f})$, to become a Hecke character on $I(\mathfrak{f})$, we require that 
\begin{equation} \label{ml}
\chi_{\text{Hecke}}(\mathfrak{a}_l)=\chi_{\text{Hecke}}\left(\mathfrak{a}_l^{m_l}\right)^{1/m_l} \quad \mbox{for } l=1,...,k,
\end{equation}
where $\chi_{\text{Hecke}}\left(\mathfrak{a}_l^{m_l}\right)^{1/m_l}$ denotes one of the $m_l$-th roots of 
$\chi_{\text{Hecke}}\left(\mathfrak{a}_l^{m_l}\right)$, which we keep fixed in the following. Now let $\psi$ be a character for the class group $I(\mathfrak{f})/P(\mathfrak{f})$. Write $\mathfrak{b}\in I(\mathfrak{f})$ uniquely in the form
$$
\mathfrak{b}=\mathfrak{a}_1^{j_1}\cdots \mathfrak{a}_k^{j_k}(\alpha)\quad \mbox{with } 0\le j_l\le m_l-1 \mbox{ for } l=1,...,k \mbox{ and } \alpha\in K. 
$$
Then 
$$
\chi_{\text{Hecke}}(\mathfrak{b})=\psi\left(\mathfrak{a}_1^{j_1}\cdots \mathfrak{a}_k^{j_k}P(\mathfrak{f})\right)\chi_{\text{Hecke}}\left(\mathfrak{a}_1^{m_1}\right)^{j_1/m_1}\cdots \chi_{\text{Hecke}}\left(\mathfrak{a}_k^{m_k}\right)^{j_k/m_k} \chi_{\text{Hecke}}((\alpha))
$$
defines a Hecke character on $I(\mathfrak{f})$. All Hecke characters on $I(\mathfrak{f})$ are in the above form. Thus, for every Hecke character $\chi_{\text{Hecke}}$ on $P(\mathfrak{f})$, there exist precisely $h_K$ extensions to $I(\mathfrak{f})$, one for each class group character $\psi$. Here we note that a change of the $m_l$-th root in \eqref{ml} amounts to a change of the class group character. We say that $\chi_{\text{Hecke}}$, defined on $I(\mathfrak{f})$, belongs to a Dirichlet character $\chi$ modulo $\mathfrak{f}$ if its restriction to $P(\mathfrak{f})$ does. We point out that class group characters themselves are Hecke characters modulo $(1)$. 

Among the Hecke characters modulo $\mathfrak{f}$, there is the principal character $\chi_{\text{Hecke},0}$ which satisfies 
$$
\chi_{\text{Hecke},0}(x)= \begin{cases} 1 & \mbox{ if } x\in I(\mathfrak{f}),\\ 0 & \mbox{ otherwise.} \end{cases}
$$ 
It belongs to the principal Dirichlet character $\chi_0$ modulo $\mathfrak{f}$. 

\subsection{Prime number theorem for Hecke characters}
In this subsection, we estimate character sums of the form 
$$
\sum\limits_{\substack{\mathfrak{p} \in \mathbb{P}\\ \mathcal{N}(\mathfrak{p})\le X}} \chi_{\text{Hecke}}(\mathfrak{p}),
$$
where $\mathbb{P}$ is the set of non-zero prime ideals in $\mathcal{O}_K$ and $\chi_{\text{Hecke}}$ is a Hecke character modulo $\mathfrak{f}$. This requires some preparations. Below we assume that $K$ is a general number field of degree $n$ and discriminant $d_K$ over $\mathbb{Q}$. We also assume that $X\ge 2$ throughout the following.

There exists a smallest (with respect to inclusion) ideal $\mathfrak{f}'$ dividing $\mathfrak{f}$ (or, equivalently, $\mathfrak{f}\subseteq \mathfrak{f}'$), such that $\chi_{\text{Hecke}}$, restricted to $I(\mathfrak{f})$, equals the restriction of a suitable Hecke character $\chi_{\text{Hecke}}'$ modulo $\mathfrak{f}'$ to $I(\mathfrak{f})$. In this case, $\mathfrak{f}'$ is called the conductor of $\chi_{\text{Hecke}}$, and we say that $\chi_{\text{Hecke}}'$ induces the character $\chi_{\text{Hecke}}$. If the conductor equals the modulus, i.e., $\mathfrak{f}'=\mathfrak{f}$, then $\chi_{\text{Hecke}}$ is referred to  as a primitive Hecke character. 

Clearly, in the above notation,
\begin{equation} \label{redprim}
\sum\limits_{\substack{\mathfrak{p} \in \mathbb{P}\\ \mathcal{N}(\mathfrak{p})\le X}} \chi_{\text{Hecke}}(\mathfrak{p})=\sum\limits_{\substack{\mathfrak{p} \in \mathbb{P}\\ \mathcal{N}(\mathfrak{p})\le X}} \chi_{\text{Hecke}}'(\mathfrak{p})+O\left(\omega(\mathfrak{f})\right),
\end{equation}
where $\omega(\mathfrak{f})$ is the number of prime ideal divisors of $\mathfrak{f}$. Since $\chi_{\text{Hecke}}'(\mathfrak{p})$ is primitive, it suffices to bound character sums over prime ideals for {\it primitive} Hecke characters.  If $\mathfrak{f}=(1)=\mathcal{O}_K$, then the relevant character sum turns into 
$$
\sum\limits_{\substack{\mathfrak{p} \in \mathbb{P}\\ \mathcal{N}(\mathfrak{p})\le X}} \chi_{\text{Hecke}}(\mathfrak{p})=\sum\limits_{\substack{\mathfrak{p} \in \mathbb{P}\\ \mathcal{N}(\mathfrak{p})\le X}} 1.
$$
In this case, the following asymptotic estimate can be obtained from a general version of the prime number theorem under GRH, \cite[Theorem 5.15]{IwKo}, via partial summation.

\begin{theorem}[Prime ideal theorem under GRH]\label{PITunderGRH} Under GRH for the Dedekind zeta function of $K$, we have
\begin{equation*}
\sum\limits_{\substack{\mathfrak{p} \in \mathbb{P}\\ \mathcal{N}(\mathfrak{p})\le X}} 1= \int\limits_2^X \frac{\mbox{\rm d}t}{\log t} + O\left(X^{1/2}\log X\right),
\end{equation*}
the implied $O$-constant depending only on $K$. 
\end{theorem}

Below we assume that $\chi_{\text{Hecke}}$ is a primitive Hecke character modulo $\mathfrak{f}$, where $\mathfrak{f}\not=(1)$. The Hecke $L$-function associated to $\chi_{\text{Hecke}}$ is defined as 
$$
L(\chi_{\text{Hecke}},s)=\sum\limits_{(0)\not= \mathfrak{a}\subseteq \mathcal{O}_K} \chi_{\text{Hecke}}(\mathfrak{a})\mathcal{N}(\mathfrak{a})^{-s}= \prod\limits_{\mathfrak{p}\in \mathbb{P}} \left(1-\chi_{\text{Hecke}}(\mathfrak{p})\mathcal{N}(\mathfrak{p})^{-s}\right)^{-1} \quad \mbox{ if } \Re(s)>1,
$$
where the sum above is taken over all non-zero integral ideals $\mathfrak{a}$ in $\mathcal{O}_K$. This function extends analytically to the whole complex plane, and the extension, also denoted $L(\chi_{\text{Hecke}},s)$, satisfies a functional equation relating 
$L(\chi_{\text{Hecke}},s)$ to $L(\overline{\chi}_{\text{Hecke}},1-s)$. In fact, $L(\chi_{\text{Hecke}},s)$  is an $L$-function of degree $n$ and conductor $|d_K|\mathcal{N}(\mathfrak{f})$, where $d_K$ is the discriminant of $K$. For the details, see \cite[pages 60 and 129]{IwKo}. Again applying \cite[Theorem 5.15]{IwKo} together with partial summation, we have the following.

\begin{theorem}[Prime number theorem for Hecke characters]\label{PNTHecke}  Let $\chi_{\text{\rm Hecke}}$ be a primitive Hecke character modulo an ideal $\mathfrak{f}\not=(1)$ in $\mathcal{O}_K$. Assume that GRH holds for the Hecke $L$-function $L(\chi_{\text{\rm Hecke}},s)$. Then
$$
\sum\limits_{\substack{\mathfrak{p} \in \mathbb{P}\\ \mathcal{N}(\mathfrak{p})\le X}} \chi_{\text{\rm Hecke}}(\mathfrak{p}) \ll X^{1/2}\log\left(X\mathcal{N}(\mathfrak{f})\right),
$$
the implied constant only depending on $K$. 
\end{theorem}    

Now, for any Hecke character modulo $\mathfrak{f}$ (not necessarily primitive), we infer the following from \eqref{redprim}, Theorems \ref{PITunderGRH} and \ref{PNTHecke} and the well-known bound $\omega(\mathfrak{f})\ll \log \mathcal{N}(\mathfrak{f})$. 
 
\begin{corollary}\label{GRHcharsum} Let $\chi_{\text{\rm Hecke}}$ be a Hecke character modulo an ideal $\mathfrak{f}$ in $\mathcal{O}_K$ which is induced by a primitive Hecke character $\chi_{\text{\rm Hecke}}'$. Assume that GRH holds for the Hecke $L$-function $L(\chi_{\text{\rm Hecke}}',s)$. Then
\begin{equation*}
    \sum\limits_{\substack{\mathfrak{p} \in \mathbb{P}\\ \mathcal{N}(\mathfrak{p})\le X}} \chi_{\rm Hecke}(\mathfrak{p})=\begin{cases}
        \int\limits_{2}^{X}\frac{\mbox{\tiny \rm d}t}{\log t} +O\left(X^{1/2}\log\left(X\mathcal{N}(\mathfrak{f})\right)\right) & \mbox{\rm if }  \chi_{\rm Hecke} \ \mbox{\rm is principal,}\\
        O\left(X^{1/2}\log\left(X\mathcal{N}(\mathfrak{f})\right)\right) & \mbox{\rm otherwise.}
    \end{cases}
\end{equation*}
\end{corollary}

\subsection{Landau's bound for sums of Hecke characters}
We also need a bound for sums of the form
$$
\sum\limits_{\substack{\mathfrak{a} \ \text{integral ideal}\\ \mathcal{N}(\mathfrak{a})\le X}} \chi_{\text{Hecke}}(\mathfrak{a}),
$$
generalizing the Polya-Vinogradov bound for Dirichlet characters. Landau proved the following result for general number fields $K$ of degree $n$ over $\mathbb{Q}$. 

\begin{theorem}[Landau] \label{Landau} Let $\chi_{\text{\rm Hecke}}$ be a Hecke character modulo an ideal $\mathfrak{f}$ in $\mathcal{O}_K$. Then 
$$
\sum\limits_{\substack{\mathfrak{a} \ \text{\rm integral ideal}\\ \mathcal{N}(\mathfrak{a})\le X}} \chi_{\rm Hecke}(\mathfrak{a}) = \begin{cases} A_KX+ O\left(\mathcal{N}(\mathfrak{f})^{1/(n+1)} \left(\log^n 2\mathcal{N}(\mathfrak{f})\right) X^{(n-1)/(n+1)}\right) & \mbox{\rm if } \chi_{\rm Hecke} \ \mbox{\rm is principal,}\\
O\left(\mathcal{N}(\mathfrak{f})^{1/(n+1)} \left(\log^n 2\mathcal{N}(\mathfrak{f})\right) X^{(n-1)/(n+1)}\right) & \mbox{\rm otherwise,}\end{cases}
$$
where $A_K$ is a positive constant depending only on the field $K$. 
\end{theorem}

\begin{proof} This is proved in \cite{Lan}.
\end{proof}

In the case of class group characters, we have $\mathfrak{f}=(1)$ and therefore deduce the following from Theorem \ref{Landau} above. 

\begin{corollary} \label{classgroupchar} Let $\psi$ be a character for the class group of $K$. Then 
$$
\sum\limits_{\substack{\mathfrak{a} \ \text{\rm integral ideal}\\ \mathcal{N}(\mathfrak{a})\le X}} \psi(\mathfrak{a}) = \begin{cases} A_KX+ O\left(X^{(n-1)/(n+1)}\right) & \mbox{\rm if } \psi \ \mbox{\rm is principal,}\\
O\left(X^{(n-1)/(n+1)}\right) & \mbox{\rm otherwise,}\end{cases}
$$
where $A_K$ is a positive constant depending only on the field $K$. 
\end{corollary}

\section{Proof of Theorem \ref{firstmainresult}}
In this section, we shall assume that $K$ is an imaginary quadratic field and prove Theorem \ref{firstmainresult}, our first main result. We handle the easier case when the class number equals 1 first. Afterwards, we extend our proof to the general case of an arbitrary class number. We shall use the following notations. 

\subsection{Notations}
Throughout this section, let $K= \mathbb{Q}(\sqrt{-d})$ be an imaginary quadratic field, where $d$ is a positive square-free integer. 
\begin{itemize}
\item Recall that $\mathcal{O}_K=\mathbb{Z}[\eta]$, where 
$$
\eta:=\begin{cases}
\sqrt{-d} & \mbox{if} -d\equiv 2,3 \bmod{4},\\ \frac{1+\sqrt{-d}}{2} & \mbox{if} -d\equiv 1 \bmod 4.
\end{cases}
$$
Thus, $\{1,\eta\}$ forms a $\mathbb{Z}$-basis of $\mathcal{O}_K$.
In particular, $\{1,\eta\}$ works as an $\mathbb{R}$-basis of $\mathbb{C}$. 
\item For any $x\in \mathbb{R}$, we define 
$$
||x||:=\mbox{min}\left\{ |x-a| : a\in \mathbb{Z} \right\},
$$ 
the distance of $x$ to the nearest integer. This function gives rise to a metric on $\mathbb{R}/\mathbb{Z}$ since it satisfies the triangle inequality
$$
||x+y||\le ||x||+||y|| \quad \mbox{for all } x,y\in \mathbb{R}.
$$
\item By $|.|$, we denote the Euclidean norm on $\mathbb{C}$. We note that
$\mathcal{N}\left(m\right)=|m|^2$ if $m\in \mathcal{O}_K$.
\item For any $x\in \mathbb{C}$, we define  
\begin{equation*}
||x||_K:=\mbox{min}\left\{ |x-a| : a\in \mathcal{O}_K\right\},
\end{equation*}
the distance of $x$ to the nearest algebraic integer in $\mathcal{O}_K$. This function gives rise to a metric on $\mathbb{C}/\mathcal{O}_K$ since it satisfies the triangle inequality
$$
||x+y||_K\le ||x||_K+||y||_K
$$ 
for all $x,y\in \mathbb{C}$.
\item We define the Euler totient function on non-zero ideals $\mathfrak{a}$ in $\mathcal{O}_K$ as 
$$
\varphi\left(\mathfrak{a}\right):=\sharp\left(\mathcal{O}_K/\mathfrak{a}\right)^{\ast}
$$
and set
$$
\varphi\left(q\right):=\varphi((q))
$$
if $q\in \mathcal{O}_K$. These notations will be kept in the case of real quadratic fields $K$ in section \ref{realquadfields}.
\item  We define the M\"obius function on the non-zero ideals $\mathfrak{a}$ in $\mathcal{O}_K$ as 
$$
\mu(\mathfrak{a}):=\begin{cases} (-1)^r & \mbox{if } \mathfrak{a} \mbox{ is a square-free ideal with exactly } r \mbox{ distinct prime ideal divisors,} \\ 0 & \mbox{otherwise}
\end{cases}
$$ 
and set
$$
\mu\left(q\right):=\mu((q))
$$
if $q\in \mathcal{O}_K$. These notations will be kept in the case of real quadratic fields $K$ in section \ref{realquadfields}.
\end{itemize}

\subsection{The case of class number $h_K=1$} \label{hK1}To prove Theorem \ref{firstmainresult},
it suffices to show that there exists a sequence of real numbers $N\ge 1$ tending to infinity such that
\begin{equation} \label{goal1}
\sum\limits_{\substack{\\p \ \text{prime element in } \mathcal{O}_K\\ N<|p|\leq 2N\\||p\alpha||_K\leq \delta}}  1 >0,
\end{equation}
where we set
\begin{equation} \label{delta1}
\delta:= N^{-1/3+\varepsilon}.
\end{equation}
Here we note that 
$$
||p\alpha||_K=|p|\min\limits_{a\in \mathcal{O}_K} \left|\alpha-\frac{a}{p}\right|  
$$
and recall that $\mathcal{N}(p)= |p|^2$. 
If $h_K=1$, then Theorem \ref{Dirichlet2} implies the existence of a sequence of elements $q\in \mathcal{O}_K$ with $\mathcal{N}(q)$ increasing to infinity for which there exists $a\in \mathcal{O}_K$ coprime to $q$ such that 
\begin{equation} \label{Diri1}
\left|\alpha -\frac{a}{q}\right|\le \frac{C}{|q|^2}.
\end{equation}
Our strategy is to take any of these $q$ and set $N:=c_0|q|^\tau$ for some suitable $c_0>0$ and $\tau >1$, to be fixed later. We then prove \eqref{goal1} for the resulting increasing and unbounded sequence of real numbers $N$, making use of \eqref{Diri1}.

Under the conditions \eqref{Diri1},
\begin{equation} \label{firstcondi}
         \frac{2CN}{|q|^2}\leq \frac{\delta}{2}
\end{equation}
 and
        \begin{equation}\label{cond1}
            pa\equiv k \bmod{q} \mbox{ for some } k\in \mathcal{O}_K \mbox{ with } (k,q)\approx 1, \ 0<|k|\leq \frac{|q|\delta}{2}, 
        \end{equation} 
we have 
\begin{equation}
    ||p \alpha||_K\leq \delta,
\end{equation} 
provided that $p\nmid q$. 
This is because 
$$\left\Vert p \alpha\right\Vert_K\le \left\Vert p \alpha- \frac{pa}{q} \right\Vert_K+\left\Vert\frac{pa}{q}\right\Vert_K\le \left|p\cdot\left(\alpha-\frac{a}{q}\right)\right|+\left\Vert\frac{pa}{q}\right\Vert_K\le \frac{2CN}{|q|^2}+\left\Vert\frac{pa}{q}\right\Vert_K$$
and 
$$
\left\Vert\frac{pa}{q}\right\Vert_K \le \frac{\delta}{2} \iff \eqref{cond1}
$$
if $p\nmid q$.
Hence, it suffices to show that
\begin{equation} \label{suff1}
S:=\sum\limits_{\substack{k\in \mathcal{O}_K\\ 0<|k|\le |q|\delta/2\\ (k,q)\approx 1}} \sum\limits_{\substack{p \text{ prime element in }\mathcal{O}_K\\ N<|p|\le 2N\\pa\equiv k \bmod{q} }} 1 >0,
\end{equation}
where we note that $p\nmid q$ if $|p|>N\ge |q|$. In fact, we shall show that
\begin{equation} \label{suff1'}
S\gg N^{4/3}.
\end{equation}
In the following, we shall tacitly assume that $k\in \mathcal{O}_K$ and $p$ is a prime element in $\mathcal{O}_K$.  

We detect the congruence condition above using Dirichlet characters modulo $q$, getting
\begin{equation}\label{S1}
S= \frac{1}{\varphi(q)}\cdot \sum\limits_{\chi \bmod q} \sum\limits_{0<|k|\le |q|\delta /2} \chi(ak^{-1})
\sum\limits_{\substack{N<|p|\le 2N}}  \chi(p).
\end{equation} 

Next, we estimate the innermost sum over $p$.  We notice that the Dirichlet character $\chi$ gives rise to a Hecke character $\chi_{\text{Hecke}}$ (defined on the fractional ideals of $\mathcal{O}_K$) satisfying $\chi_{\text{Hecke}}((p))=\chi(p)$ if $\chi$ is trivial on all units in $\mathcal{O}_K^{\ast}$. Otherwise, if $\chi(u)\not=1$ for some unit $u\in \mathcal{O}_K^{\ast}$, we have
\begin{equation} \label{canc}
\sum\limits_{\tilde{p} \approx p} \chi(\tilde{p})= \chi(p) \sum\limits_{\substack{u\in \mathcal{O}_K^* }}\chi(u) = 0
\end{equation}
for any $p\in \mathcal{O}_K$, the sum on the left-hand side running over all $\tilde{p}$ which are associates to $p$. Hence, in this case, 
\begin{equation*}
\sum\limits_{\substack{N<|p|\le 2N}}  \chi(p) =0.
\end{equation*}
If $\chi(u)=1$ for all units $u\in \mathcal{O}_K^{\ast}$, then applying Corollary \ref{GRHcharsum}, we get  
\begin{equation}\label{primesums2}
 \begin{split}
     \sum\limits_{\substack{N<|p|\le 2N}}  \chi(p)=& w\cdot \sum\limits_{\substack{\mathfrak{p}\in \mathbb{P}\\ N^2<\mathcal{N}(\mathfrak{p})\le 4N^2}}\chi_{\text{Hecke}}(\mathfrak{p})= \delta(\chi)w\cdot \int\limits_{N^2}^{4N^2} \frac{\mbox{d}t}{\log t} + O(N\log N), 
 \end{split}
 \end{equation}
where $w$ is the order of the unit group $\mathcal{O}_K^{\ast}$ and 
\begin{equation} \label{anotherdelta}
\delta(\chi)=\begin{cases} 1 & \mbox{ if } \chi=\chi_0 \mbox{ is principal,}\\ 0 & \mbox{ otherwise.} \end{cases}
\end{equation}

Now, plugging \eqref{primesums2} into \eqref{S1}, we get
\begin{equation}\label{all1}
S=  \frac{w}{\varphi(q)} \cdot \int\limits_{N^2}^{4N^2} \frac{\mbox{d}t}{\log t} \cdot  \sum\limits_{\substack{0<|k|\le |q|\delta/2 \\
(k,q)\approx 1}} 1 +
O\left(\frac{N\log N}{\varphi(q)}\cdot \sum\limits_{\chi \bmod q}
\left| \sum\limits_{0<|k|\le|q|\delta/2 } \chi(ak^{-1}) \right|\right).
\end{equation}
Using Cauchy-Schwarz and orthogonality relations for Dirichlet characters, we have
\begin{equation} \label{CS1}
\begin{split}
\sum\limits_{\chi\bmod q} \left|\sum\limits_{0<|k|\le |q|\delta/2} \chi(ak^{-1}) \right| =& \sum\limits_{\chi\bmod q} \left|\sum\limits_{0<|k|\le |q|\delta/2} \overline{\chi}(k) \right|\\
\le &
\varphi(q)^{1/2} \left(\sum\limits_{\chi \bmod{q}} \left| \sum\limits_{0<|k|\le |q|\delta/2} \overline{\chi}(k)\right|^2\right)^{1/2}\\
= & \varphi(q)^{1/2} \left(\sum\limits_{\chi \bmod{q}} \ \sum\limits_{0<|k_1|,|k_2|\le  |q|\delta/2} \overline{\chi}(k_1)\chi(k_2) \right)^{1/2}\\=& \varphi(q)^{1/2}\left(\varphi(q)\sum\limits_{\substack{0<|k_1|,|k_2|\le  |q|\delta/2\\k_1\equiv k_2 \bmod{q}}}1\right)^{1/2}\\
= & \varphi(q)\left(\sum\limits_{0<|k|\le |q|\delta/2}1\right)^{1/2}\\
\ll & \varphi(q) |q|\delta
\end{split} 
\end{equation}
where the second-last line arrives by the observation that $k_1\equiv k_2\bmod{q}$ and $0<|k_1|,|k_2|\le |q|\delta/2$ imply $k_1=k_2$ if $\delta\le 1$. So the error term in \eqref{all1} is $\ll |q|\delta N\log N$.

We detect the coprimality condition $(k,q)\approx 1$ in the main term via the relation
\begin{equation}\label{mufnct}
    \sum\limits_{\substack{\mathfrak{d}|\mathfrak{a}\\ \mathfrak{d}  \text{ integral ideal}}} \mu(\mathfrak{d})=\begin{cases} 1 & \mbox{ if } \mathfrak{a}=(1),\\ 0 & \mbox{ otherwise,} \end{cases}
\end{equation}
obtaining
\begin{equation} \label{Mobius}
\begin{split}
\sum\limits_{\substack{0<|k|\le |q|\delta/2\\ (k,q)\approx 1}} 1= & \sum\limits_{0<|k|\le |q|\delta/2}\ \sum\limits_{\substack{\mathfrak{d}|\left((k),(q)\right)}} \mu(\mathfrak{d}) \\
= &  \sum\limits_{\substack{\mathfrak{d}|(q)}} \mu(\mathfrak{d}) \cdot \sum\limits_{\substack{0<|k|\le |q|\delta/2\\ \mathfrak{d}|(k)}} 1\\ 
= & \sum\limits_{\mathfrak{d}|(q)} \mu(\mathfrak{d}) \sum\limits_{0<|c|\le |q|\delta/(2|g|)} 1,
\end{split}
\end{equation}
where $g$ is any generator of $\mathfrak{d}$.
The innermost sum over $c$ in the last line of \eqref{Mobius} counts lattice point of the form $n_1+n_2\eta$ with $(n_1,n_2)\in \mathbb{Z}$ in a circle of radius $|q|\delta/(2|g|)$. The area of a fundamental parallelogram of this lattice equals 
$$
\Im(\eta)=\begin{cases}
\sqrt{d} & \mbox{if } -d\equiv 2,3 \bmod{4},\\ \sqrt{d}/2 & \mbox{if } -d\equiv 1\bmod{4}.
\end{cases}
$$
It follows that the number of lattice points in question equals
$$
\sum\limits_{0<|c|\le |q|\delta/(2|g|)} 1=\frac{\pi|q|^2\delta^2}{4\Im(\eta)\mathcal{N}(\mathfrak{d})}+O\left(|q|\delta\right),
$$
where we note that $|g|^2=\mathcal{N}(\mathfrak{d})$. 

Using the relation 
\begin{equation}\label{relation1}
\sum\limits_{\mathfrak{d}|(q)} \frac{\mu(\mathfrak{d})}{\mathcal{N}(\mathfrak{d})}=\frac{\varphi(q)}{\mathcal{N}(q)}
\end{equation}
and the well-known bounds
\begin{equation}\label{relation2}
\sum\limits_{\mathfrak{d}|(q)} 1 \ll_{\varepsilon} \mathcal{N}(q)^{\varepsilon} \quad \mbox{and} \quad \mathcal{N}(q)^{1-\varepsilon}\ll_{\varepsilon} \varphi(q) \quad \mbox{for any } \varepsilon>0,
\end{equation}
it follows that
\begin{equation}\label{main1}
\begin{split}
\frac{1}{\varphi(q)}  \sum\limits_{\substack{0<|k|\le |q|\delta/2\\ (k,q)\approx 1}} 1= &
\frac{1}{\varphi(q)} \sum\limits_{\mathfrak{d}|(q)} \mu(\mathfrak{d}) \cdot \left( \frac{\pi|q|^2\delta^2}{4\Im(\eta){\mathcal{N}(\mathfrak{d})}}+O\left(|q|\delta\right)\right)\\
= & \frac{1}{\varphi(q)}\cdot \frac{\pi|q|^2\delta^2}{4\Im(\eta)}\sum\limits_{\mathfrak{d}|(q)}\frac{\mu(\mathfrak{d})}{\mathcal{N}(\mathfrak{d})}+O\left(\frac{|q|\delta}{\varphi(q)}\sum\limits_{\mathfrak{d}|(q)}1\right)\\
= &\frac{\pi\delta^2}{4\Im(\eta)}+ O\left(\delta |q|^{4\varepsilon -1}\right).
\end{split}
\end{equation}

Combining \eqref{all1}, \eqref{CS1} and \eqref{main1}, we get
\begin{equation*}
S = \frac{w \pi \delta^2}{4\Im(\eta)}\cdot \int\limits_{N^2}^{4N^2}\frac{\mbox{d}t}{\log t}+O\left(\delta |q|^{4\varepsilon-1}N^2+|q|\delta N\log N\right).
\end{equation*}
Now we fix $N$ in such a way that 
$$\frac{2CN}{|q|^2}=\frac{\delta}{2},$$ 
so that the inequality \eqref{firstcondi} is satisfied. 
In view of \eqref{delta1}, this amounts to setting
$$
N:=\left|\frac{q}{2\sqrt{C}}\right|^{2/(4/3-\varepsilon)},
$$
which is equivalent to
$$
|q|=2\sqrt{C}N^{2/3-\varepsilon/2}.
$$
Hence, we obtain
$$S=\frac{w \pi}{4 \Im(\eta)N^{2/3-2\varepsilon}}\int\limits_{N^2}^{4N^2}\frac{\mbox{d}t}{\log t}+O\left(N^{4/3+\varepsilon}\right)$$
if $\varepsilon$ is small enough. Now, for sufficiently large $N$, the main term on the right hand side supersedes the error term and \eqref{suff1} thus holds, which completes the proof of Theorem \ref{firstmainresult} in case when $h_K=1$.

\subsection{The case of an arbitrary class number $h_K$}\label{arb}
In this subsection, we follow the treatment of general class numbers of imaginary quadratic function fields in \cite[section 4]{BaMo}. Since many arguments are parallel in the number field case, we omit a few details and refer to \cite{BaMo} instead.

At first, we point out the extra difficulties in the case of class number greater than 1. Since the ring of integers of an imaginary quadratic number field of class number greater than $1$ is not a UFD, a greatest common divisor of two elements $a,b$ of $\mathcal{O}_K$ does not necessarily exist. Furthermore, it may happen that $a,b$ are coprime as elements of $\mathcal{O}_K$, i.e., they don't share irreducible factors, but are not coprime in the sense of ideals, i.e. gcd$((a),(b))\not\approx (1)$. (For example, this happens when $(a)$ and $(b)$ have prime ideal factorizations of the form $(a)=\mathfrak{p}_1\mathfrak{p}_2$ and $(b)=\mathfrak{p}_1\mathfrak{p}_3$, where $\mathfrak{p}_2\not=\mathfrak{p}_3$ and $\mathfrak{p}_1,\mathfrak{p_2},\mathfrak{p}_3$ are all non-principal.) This issue causes some difficulties when we apply the Dirichlet approximation theorem: In a fraction $a/q$ of elements of $\mathcal{O}_K$, it is possible to cancel out common irreducible factors, yet the ideals generated by the reduced numerator and denominator may not be coprime. To make the arguments in subsection \ref{hK1} work for the general setting of arbitrary class numbers, we need to cancel common ideal factors from $(a)$ and $(q)$. Therefore, we need to work with the ideal  $\mathfrak{q}=\mbox{gcd}((a),(q))^{-1}(q)$ in place of $(q)$. To generate an infinite sequence of $N$'s for which the desired lower bound holds, we now need to ensure that Dirichlet's approximation theorem holds with the extra condition $\mathcal{N}(\mathfrak{q})\rightarrow \infty$ included (which clearly works if $\mathcal{O}_K$ is a PID). Hence, we require the following sharpened version of Dirichlet's approximation theorem for imaginary quadratic number fields.

\begin{theorem} \label{Dirichlet3} 
Let $K$ be an imaginary quadratic number field. Then there exists a constant $C$ depending at most on $K$ such that for every $\alpha\in \mathbb{C}\setminus K$, there exists an infinite sequence of pairs  $(a,q)\in\mathcal{O}_K\times(\mathcal{O}_K\setminus\{0 \})$ such that $\mathcal{N}(\mbox{gcd}((a),(q))^{-1}(q))$ increases to infinity and 
\begin{equation}\label{Gin}
    \left|\alpha-\frac{a}{q}\right| \leq \frac{C}{\mathcal{N}(q)}.
\end{equation}
\end{theorem} 
\begin{proof} 
Using Theorem \ref{Dirichlet2}, there exists a sequence of pairs $(a_n,q_n)\in\mathcal{O}_K\times(\mathcal{O}_K\setminus\{0 \})$ satisfying the inequalities \eqref{Gin} and $\mathcal{N}(q_n)<\mathcal{N}(q_{n+1})$ for all $n\in \mathbb{N}$. Plugging the pairs $(a,q)=(a_n,q_n), \, (a_{n+1},q_{n+1})$ into \eqref{Gin} and applying the triangle inequality, we deduce that
\begin{equation*}
\left|\frac{a_n}{q_n}-\frac{a_{n+1}}{q_{n+1}}\right| < \frac{2C}{\mathcal{N}(q_n)} =\frac{2C}{|q_n|^2} 
\end{equation*} 
and hence 
$$
\left|a_nq_{n+1}-a_{n+1}q_n\right| < 2C\cdot \frac{|q_{n+1}|}{|q_n|}
$$
upon multiplying by $|q_nq_{n+1}|$. Squaring both sides, we obtain
$$
\mathcal{N}(a_nq_{n+1}-a_{n+1}q_n) < (2C)^2\cdot \frac{\mathcal{N}(q_{n+1})}{\mathcal{N}(q_n)}.
$$
Since 
$$
\mathcal{N}\left(\gcd((a_{n+1}),(q_{n+1}))\right)\le\mathcal{N}(a_nq_{n+1}-a_{n+1}q_n),
$$
it follows that
$$
\mathcal{N}\left(\gcd((a_{n+1}),(q_{n+1}))^{-1}(q_{n+1})\right)>(2C)^{-2}\mathcal{N}(q_n).
$$
This shows that the sequence $\mathcal{N}(\gcd((a_n),(q_n))^{-1}(q_n))$ diverges to infinity, which completes the proof. 
\end{proof}
Another issue that comes up in the case of class number greater than 1 is that we need to pick out principal prime ideals (i.e., ideals generated by prime elements in $\mathcal{O}_K$) from the set of all prime ideals. This can be conveniently done using class group characters. 

We start as in the previous section, aiming to show that \eqref{goal1} holds
for an infinite sequence of positive real numbers $N$ tending to infinity. Here we use Theorem \ref{Dirichlet3} to approximate $\alpha$ in the form
\begin{equation}\label{eq2}
    \left|\alpha-\frac{a}{q}\right|\le\frac{C}{|q|^2}
\end{equation}
with $(a,q)\in \mathcal{O}_K\times (\mathcal{O}_K\setminus\{0\})$ and some constant $C>0.$ Throughout the following, we shall write 
$$
\mathfrak{D}:=\mbox{gcd}((a),(q))\quad \mbox{and} \quad \mathfrak{q}:=\mathfrak{D}^{-1}(q).
$$ 
We note that \eqref{eq2} implies that
\begin{equation}\label{eq3}
     \left|\alpha-\frac{a}{q}\right|\le\frac{C}{\mathcal{N}(\mathfrak{q})}.
\end{equation}
The parameter $N$ will later be fixed as a function of $\mathcal{N}(\mathfrak{q})$, where $\log \mathcal{N}(\mathfrak{q})\le \log N$. We define $\delta$ as in the previous subsection by \eqref{delta1}.
Now we observe that \eqref{eq3} implies $$||\alpha p||_K\le \delta$$ for a prime element $p\in \mathcal{O}_K$, provided that $(p)\nmid \mathfrak{q}$, 
\begin{equation}\label{eq5}
    \frac{2CN}{\mathcal{N}(\mathfrak{q})}\le \frac{\delta}{2}
\end{equation}
and
$$
pa\equiv k \bmod{q} \mbox{ for some } k\in\mathcal{O}_K \mbox{ with } \mathfrak{C}=\mathfrak{D} \mbox{ and } 0<|k|\le\frac{|q|\delta}{2},
$$
where we set 
$$
\mathfrak{C}:=\mbox{gcd}((k),(q)).
$$
Hence, to prove Theorem \ref{firstmainresult}, it is enough to show that 
\begin{equation}\label{eq6}
    S:=\sum\limits_{\substack{0<|k|\le |q|\delta/2\\\mathfrak{C}=\mathfrak{D}}}~\sum\limits_{\substack{p \text{ prime element }\\ (p)\nmid \mathfrak{q} \\ N<|p|\le 2N\\ pa\equiv k \bmod{q}}} 1 >0.
\end{equation}
(Again, we shall establish the stronger inequality $S\gg N^{4/3}$.) Since $(a)$ and $(q)$ are generally not coprime ideals, we need to turn the congruence $pa\equiv k \bmod{q}$ into a multiplicative congruence $pak^{-1} \equiv^{\ast} 1 \bmod \mathfrak{q}$ before detecting it using Dirichlet characters. (For the details of this process, see \cite[equation (4.8)]{BaMo}.) Note that the fractional ideal $(ak^{-1})$ is coprime to $\mathfrak{q}$, whereas the integral ideals $(a)$ and $(k)$ are not necessarily. As a result, $S$ turns into the form
\begin{equation*}
    S=\sum\limits_{\substack{0<|k|\le |q|\delta/2\\ \mathfrak{C}=\mathfrak{D}}}~\sum\limits_{\substack{p \text{ prime element }\\ (p)\nmid \mathfrak{q}\\ N<|p|\le 2N\\pak^{-1}\equiv^* 1 \bmod \mathfrak{q}}} 1.
 \end{equation*}

Next, we use orthogonality relations for Dirichlet characters modulo $\mathfrak{q}$ to write the above sum as
\begin{equation*}
 S=\frac{1}{\varphi(\mathfrak{q})}\sum\limits_{\chi \bmod \mathfrak{q}} \sum\limits_{\substack{0<|k|\le |q|\delta/2\\\mathfrak{C}=\mathfrak{D}}} \chi((ak^{-1}))\sum\limits_{\substack{ p \text { prime element }\\ N<|p|\le 2N}}\chi(p).
 \end{equation*}
We denote by $\mathbb{P}_0$ the set of non-zero principal prime ideals. These are exactly the prime ideals which are generated by prime elements. As in the previous subsection, we observe that 
$$
\sum\limits_{\substack{ p \text { prime element }\\ N<|p|\le 2N}}\chi(p) = \begin{cases} 0 & \mbox{ if } \chi(u)\not=1 \mbox{ for some unit } u\in \mathcal{O}_K^{\ast},\\ 
w  \cdot \sum\limits_{\substack{\mathfrak{p}\in \mathbb{P}_0\\ N^2<\mathcal{N}(\mathfrak{p})\le 4N^2}}\chi_{\text{Hecke}}(\mathfrak{p}) & \mbox{ otherwise,} \end{cases}
$$
where $w$ is the number of units in $\mathcal{O}_K$ and $\chi_{\text{Hecke}}$ is a Hecke character modulo $\mathfrak{q}$ with trivial infinity part belonging to $\chi$. The principality of $\mathfrak{p}$ can be picked out using orthogonality relations for class group characters $\psi$, getting 
$$
\sum\limits_{\substack{\mathfrak{p}\in \mathbb{P}_0\\ N^2<\mathcal{N}(\mathfrak{p})\le 4N^2}}\chi_{\text{Hecke}}(\mathfrak{p})= \frac{1}{h_K} \sum\limits_{\psi\in \hat{\mathcal{C}}} \sum\limits_{\substack{\mathfrak{p}\in \mathbb{P}\\ N^2<\mathcal{N}(\mathfrak{p})\le 4N^2}} \psi\chi_{\text{Hecke}}(\mathfrak{p}), 
$$
where $\hat{\mathcal{C}}$ is the character group for the class group $\mathcal{C}$ of $K$ and $\mathbb{P}$ is the set of {\it all} non-zero prime ideals. Using Corollary \ref{GRHcharsum}, we deduce that 
$$
\sum\limits_{\substack{ p \text { prime element }\\ N<|p|\le 2N}}\chi(p)=\delta(\chi) \cdot \frac{w}{h_K} \cdot \int\limits_{N^2}^{4N^2} \frac{\mbox{d}t}{\log t} + O\left(N\log N\right)
$$
under GRH, with $\delta(\chi)$ defined as in \eqref{anotherdelta}. Hence, altogether we obtain
  \begin{equation}\label{SS2}
      S=  \frac{w}{h_k\varphi(\mathfrak{q})} \cdot \int\limits_{N^2}^{4N^2} \frac{\mbox{d}t}{\log t} \cdot  \sum\limits_{\substack{0<|k|\le |q|\delta/2 \\
\mathfrak{C}=\mathfrak{D}}} 1 +
O\left(\frac{N\log N }{\varphi(\mathfrak{q})}\cdot \sum\limits_{\chi \bmod{\mathfrak{q}}}
\left| \sum\limits_{\substack{0<|k|\le|q|\delta/2\\ \mathfrak{C}=\mathfrak{D} }} \chi(ak^{-1}) \right|\right).
  \end{equation}
For an estimation of the double sum in the error term, we refer the reader to \cite[subsection 4.9]{BaMo}. This calculation generalizes the estimation of the corresponding double sum in \eqref{CS1} in the previous subsection and gives
\begin{equation}\label{errorterm}
    \sum\limits_{\chi \bmod{\mathfrak{q}}}
\left| \sum\limits_{\substack{0<|k|\le|q|\delta/2\\ \mathfrak{C}=\mathfrak{D} }} \chi(ak^{-1}) \right| \ll \varphi(\mathfrak{q})\mathcal{N}(\mathfrak{q})^{1/2}\delta.
\end{equation}
In contrast to the situation in the previous subsection, we cannot immediately separate $a$ and $k$ because the ideals $(a)$ and $(k)$ may not be coprime to $\mathfrak{q}$, as remarked earlier in this subsection. Nevertheless, an application of the Cauchy-Schwarz inequality and orthogonality relations for Dirichlet characters leads to a congruence of the form $k_1k_2^{-1}\equiv^{*} 1 \bmod{\mathfrak{q}}$ which no longer contains $a$. As we demonstrated in \cite[subsection 4.9]{BaMo}, this is equivalent to $k_1\equiv k_2 \bmod{q}$. Now we can apply the same argument as in the previous subsection to show that this congruence implies equality of $k_1$ and $k_2$ under the relevant summation conditions, thus establishing \eqref{errorterm}. We point out that in \cite[subsection 4.9]{BaMo}, we were led to the corresponding sum over all Hecke characters modulo $\mathfrak{q}$ with trivial infinity part in place of a sum of all Dirichlet characters modulo $\mathfrak{q}$, but this amounts to the same since again the inner-most sum over $k$ vanishes if $\chi$ is not trivial on the units.  

We write the sum in the main term as 
\begin{equation} \label{mainterm}
\begin{split}
\sum\limits_{\substack{0<|k|\le |q|\delta/2\\
\mathfrak{C}=\mathfrak{D}}} 1= & \sum\limits_{\substack{0<\mathcal{N}(k)\le \mathcal{N}(q)\delta^2/4\\ \mathfrak{D}|(k)\\ \mbox{gcd}(\mathfrak{D}^{-1}(k),\mathfrak{q})=1}}1 \\
= & \sum\limits_{\mathfrak{d}|\mathfrak{q}}\mu(\mathfrak{d})\sum\limits_{\substack{0<\mathcal{N}(k)\le \mathcal{N}(q)\delta^2/4\\ \mathfrak{Dd}|(k)}}1,
\end{split}
\end{equation}
where we use \eqref{mufnct} to obtain the last line. Further, writing the principal ideal $(k)$ as $(k)=\mathfrak{Dda}$ and using $\mathcal{N}(q)=\mathcal{N}(\mathfrak{Dq})$, we transform the innermost sum in the last line above into
       \begin{equation}\label{countingideals0}
\sum\limits_{\substack{0<\mathcal{N}(k)\le \mathcal{N}(q)\delta^2/4\\ \mathfrak{Dd}|(k)}} 1 =\sum\limits_{\substack{\mathfrak{a} \text{ integral ideal}\\ \mathfrak{Dda}\text{ principal}\\ 0<\mathcal{N}(\mathfrak{a})\le \mathcal{N}(\mathfrak{q})\delta^2/(4\mathcal{N}(\mathfrak{d}))}}1.
         \end{equation}
       We pick out the condition of $\mathfrak{Dda}$ being principal again using orthogonality relations for class group characters, getting 
       \begin{equation}\label{countingideals1}
\sum\limits_{\substack{\mathfrak{a} \text{ integral ideal}\\ \mathfrak{Dda} \text{ principal}\\ 0<\mathcal{N}(\mathfrak{a})\le \mathcal{N}(\mathfrak{q})\delta^2/(4\mathcal{N}(\mathfrak{d}))}}1 = \frac{1}{h_K}\sum\limits_{\psi \in \hat{\mathcal{C}}}\psi(\mathfrak{Dd})\sum\limits_{\substack{\mathfrak{a} \text{ integral ideal}\\ 0<\mathcal{N}(\mathfrak{a})\le \mathcal{N}(\mathfrak{q})\delta^2/(4\mathcal{N}(\mathfrak{d}))}}\psi(\mathfrak{a}),
       \end{equation}
where $\hat{\mathcal{C}}$ is the character group of the class group $\mathcal{C}$ of $K$. 
By Corollary \ref{classgroupchar}, we have
       \begin{equation}\label{countingideals2}
        \sum\limits_{\substack{\mathfrak{a} \text{ integral ideal}\\ 0<\mathcal{N}(\mathfrak{a})\le X}}\psi(\mathfrak{a}) = \begin{cases}
A_KX+O\left(X^{1/3}\right) & \mbox{ if } \psi \mbox{ is principal,}\\ O\left(X^{1/3}\right) & \mbox{ otherwise} \end{cases}
       \end{equation}
for $X\ge 1$, where $A_K$ is a positive constant depending only on the number field $K$. Now combining \eqref{countingideals0}, \eqref{countingideals1} and \eqref{countingideals2}, we deduce that
       \begin{equation}\label{countingideals3}
       \sum\limits_{\substack{0<\mathcal{N}(k)\le \mathcal{N}(q)\delta^2/4\\ \mathfrak{Dd}|(k)}} 1 = \frac{A_K}{h_K}\cdot \frac{\mathcal{N}(\mathfrak{q})\delta^2}{4\mathcal{N}(\mathfrak{d})} +O\left(\frac{\mathcal{N}(\mathfrak{q})^{1/3}\delta^{2/3}}{\mathcal{N}(\mathfrak{d})^{1/3}}\right).
       \end{equation}

Combining \eqref{mainterm} and \eqref{countingideals3}, and using the relations \eqref{relation1} and \eqref{relation2}, we get
\begin{equation}\label{SS20}
\sum\limits_{\substack{0<|k|\le |q|\delta/2\\
\mathfrak{C}=\mathfrak{D}}} 1= \frac{A_K}{4h_K}\cdot \varphi(\mathfrak{q})\delta^2 + O\left(\mathcal{N}(\mathfrak{q})^{1/3+\varepsilon}\delta^{2/3}\right).
  \end{equation}
Plugging \eqref{errorterm} and \eqref{SS20} into \eqref{SS2}, and recalling \eqref{relation2}, it follows that
\begin{equation} \label{Sfinal}
 S=  \frac{wA_K\delta^2}{4h_K^2} \cdot \int\limits_{N^2}^{4N^2} \frac{\mbox{d}t}{\log t} 
+
O\left(\frac{N^2\delta^{2/3}}{\mathcal{N}(\mathfrak{q})^{2/3-2\varepsilon}} + N(\log N)\mathcal{N}(\mathfrak{q})^{1/2+\varepsilon}\delta \right).
\end{equation}

Now we fix $N$ in such a way that 
$$
\frac{2CN}{\mathcal{N}(\mathfrak{q})}= \frac{\delta}{2},
$$
so that the inequality \eqref{eq5} is satisfied. In view of our choice $\delta=N^{-1/3+\varepsilon}$, this amounts to setting
$$
N:=\left|\frac{\mathcal{N}(\mathfrak{q})}{4C}\right|^{1/(4/3-\varepsilon)},
$$
which is equivalent to
$$
\mathcal{N}(\mathfrak{q})=4CN^{4/3-\varepsilon}.
$$
Hence, we obtain
$$
S= \frac{wA_K}{4h_K^2}\cdot \frac{1}{N^{2/3-2\varepsilon}}\cdot\int\limits_{N^2}^{4N^2} \frac{\mbox{d}t}{\log t}+O\left(N^{4/3+11\varepsilon/6}\log N\right)
$$
if $\varepsilon$ is small enough. For sufficiently large $N$, the main term on the right-hand side supersedes the error term and \eqref{eq6} thus holds, which completes the proof of Theorem \ref{firstmainresult}.

\section{Proof of Theorem \ref{secondmainresult}} \label{realquadfields}
In this section, we shall assume that $K$ is a real quadratic field and prove Theorem \ref{secondmainresult}, our second main result, for this case. Again, we handle the easier case when the class number equals 1 first. Afterwards, we extend our proof to the general case of an arbitrary class number.  

The main difficulty in the real quadratic case lies in the infinitude of the unit group $\mathcal{O}_K^{\ast}$. In particular, every non-zero element of $\mathcal{O}_K$ has infinitely many associates, and therefore, a sum as in \eqref{suff1} does not make sense anymore. This makes it necessary to restrict the relevant elements of $\mathcal{O}_K$ carefully. For the same reason, we do not get a cancellation as in \eqref{canc} any longer. As as result, it is not possible to restrict ourselves to Dirichlet characters which are trivial on the units. Instead, we need to work with all Dirichlet characters, and thus the Hecke characters belonging to them may no longer have trivial infinity part. These issues are resolved in this section. It is conceivable that a similar treatment can be applied to the function field case, which would improve the results obtained by the first- and third-named authors in \cite{BaMo2}. In the said paper, they used Vaughan's identity and exponential sums instead of Hecke $L$-functions to treat real quadratic function fields.

\subsection{Preliminaries about real quadratic fields}
We recall that by the Dirichlet unit theorem for the case of real quadratic fields $K$, the unit group $\mathcal{O}_K^{\ast}$ is generated by a fundamental unit $\epsilon>1$ and $-1$. We note that $\sigma_2(\epsilon)=\pm \epsilon^{-1}$ since
$$
|\epsilon \sigma_2(\epsilon)|=|\sigma_1(\epsilon)\sigma_2(\epsilon)|=\mathcal{N}(\epsilon)=1,
$$
where $\sigma_1$ and $\sigma_2$ are the two embeddings of $K$, the identity and conjugation, respectively.  To restrict generators of principal ideals suitably, we shall use the following lemma. 

\begin{lemma} \label{sigma12}
Let $\mathfrak{a}\subseteq\mathcal{O}_K$ be a principal non-zero ideal. Then there exists a generator $a$ of $\mathfrak{a}$, unique up to the sign, such that
$$
\epsilon^{-1}|\sigma_1(a)|<|\sigma_2(a)|\le \epsilon|\sigma_1(a)|,
$$
where $\epsilon$ is the fundamental unit in $\mathcal{O}_K$. In this case, we have
$$
\epsilon^{-1/2}\mathcal{N}(\mathfrak{a})^{1/2}\le |\sigma_i(a)|\le \epsilon^{1/2} \mathcal{N}(\mathfrak{a})^{1/2} \quad \mbox{for } i=1,2. 
$$ 
\end{lemma}
\begin{proof}
All generators of $\mathfrak{a}$ are of the form $\pm\epsilon^na_0$, where $a_0$ is a fixed generator, and $n$ runs over the integers. For every $n\in\mathbb{Z}$, we have 
$$
\sigma_1(\epsilon^na_0)=\epsilon^n\sigma_1(a_0)
$$
and 
$$
\sigma_2(\epsilon^na_0)=\pm \epsilon^{-n}\sigma_2(a_0).
$$
Suppose that 
$$
|\epsilon^\nu\sigma_1(a_0)|=|\epsilon^{-\nu}\sigma_2(a_0)|
$$
for some real number $\nu\in\mathbb{R}$.
This real number $\nu$ is uniquely given by 
$$
\nu=\frac{\log|\sigma_2(a_0)|-\log|\sigma_1(a_0)|}{2\log\epsilon}.
$$
Let $m$ be the unique integer satisfying $-1/2<\nu-m\le1/2$. Then we obtain
$$
\epsilon^{-1}<\frac{|\sigma_2(\epsilon^{m}a_0)|}{|\sigma_1(\epsilon^ma_0)|} =\epsilon^{-2m}\frac{|\sigma_2(a_0)|}{|\sigma_1(a_0)|}=\epsilon^{2(\nu-m)}\le\epsilon.
$$
Choosing $a=\epsilon^ma_0$, the above inequality becomes
$$
\epsilon^{-1}|\sigma_1(a)|<|\sigma_2(a)|\le \epsilon|\sigma_1(a)|.
$$
This inequality is not satisfied if $a$ is replaced by $\pm \epsilon^na$ for any $n\in \mathbb{Z}\setminus \{0\}$. Therefore, $a$ is unique up to sign. Moreover,  
since $\mathcal{N}(\mathfrak{a})=|\sigma_1(a)\sigma_2(a)|$, from the above considerations, we deduce the following inequalities
$$
\epsilon^{-1/2}\mathcal{N}(\mathfrak{a})^{1/2}\le |\sigma_i(a)|\le \epsilon^{1/2} \mathcal{N}(\mathfrak{a})^{1/2} \quad \mbox{for } i=1,2. 
$$ 
\end{proof}

\subsection{The case of class number $h_K=1$} \label{hK1new}

\subsubsection{Reduction to prime elements in residue classes}
If $h_K=1$, then Theorem \ref{Dirichlet1} implies the existence of a sequence of elements $q\in \mathcal{O}_K$ with $\mathcal{N}(q)$ increasing to infinity for which there exists $a\in \mathcal{O}_K$ coprime to $q$ such that 
\begin{equation} \label{Diri}
\left|x_i-\frac{\sigma_i(a)}{\sigma_i(q)}\right| \le \frac{C}{\mathcal{N}(q)} \quad \mbox{ for } i=1,2.
\end{equation}
Again, our strategy is to take any of these $q$ and set 
\begin{equation} \label{Ndef}
N:=\mathcal{N}(q)^\tau
\end{equation}
for a suitable $\tau\in (1,2)$. In this way, we get an increasing sequence of $N$'s tending to infinity. For these $N$'s, we will prove that there is a prime element $p$ with $N<\mathcal{N}(p)\le 2N$ and $a\in \mathcal{O}_K$ satisfying the conditions in Theorem \ref{secondmainresult}, thus establishing the desired result. 

In the following, we set
\begin{equation} \label{Deltadef}
\Delta:=\frac{C}{\mathcal{N}(q)}
\end{equation}
so that
\begin{equation} \label{Delta}
\left|x_i-\frac{\sigma_i(a)}{\sigma_i(q)}\right| \le \Delta \quad \mbox{ for } i=1,2
\end{equation}
by \eqref{Diri}. 
We aim to count prime ideals $\mathfrak{p}$ with norm $\mathcal{N}(\mathfrak{p})\in (N,2N]$ such that 
\begin{equation} \label{aim}
\left| x_i-\frac{\sigma_i(b)}{\sigma_i(p)}\right|\le 2\Delta \quad \mbox{for } i=1,2 
\end{equation}
for some generator $p\in \mathcal{O}_K$ of $\mathfrak{p}$ and a suitable $b\in\mathcal{O}_K$. In view of \eqref{Delta}, \eqref{aim} is satisfied if 
\begin{equation}\label{SigmaApprox}
\left|\frac{\sigma_i(a)}{\sigma_i(q)}-\frac{\sigma_i(b)}{\sigma_i(p)}\right| \le \Delta \quad \mbox{for } i=1,2.    
\end{equation}
We see that the choice 
\begin{equation}\label{choice}
\tau:=\frac{1}{2/3-\varepsilon}
\end{equation}
in \eqref{Ndef} yields Theorem \ref{secondmainresult}, provided we find at least one prime ideal $\mathfrak{p}$ satisfying the above if $\mathcal{N}(q)$ is large enough.  

Using Lemma \ref{sigma12}, we may assume without loss of generality that 
$$
\epsilon^{-1/2}\mathcal{N}(q)^{1/2}\le |\sigma_i(q)| \le \epsilon^{1/2}\mathcal{N}(q)^{1/2} \quad \mbox{for } i=1,2
$$
and 
\begin{equation}\label{sigmaqrel}
\epsilon^{-1}|\sigma_1(p)|<|\sigma_2(p)|\le \epsilon|\sigma_1(p)|
\end{equation}
so that 
\begin{equation} \label{sigmaibound}
\epsilon^{-1/2}\mathcal{N}(p)^{1/2}\le |\sigma_i(p)|\le \epsilon^{1/2}\mathcal{N}(p)^{1/2} \quad \mbox{for } i=1,2.
\end{equation}
(Otherwise, multiply both the numerator and denominator with a suitable unit $u$, respectively.)
Under this assumption, \eqref{SigmaApprox} is satisfied if 
\begin{equation} \label{Ybound}
\left|\sigma_i(a)\sigma_i(p) - \sigma_i(b)\sigma_i(q)\right|=\left|\sigma_i(ap-bq)\right| \le Y \quad \mbox{for } i=1,2,     
\end{equation}
where
$$
Y=\epsilon^{-1}\mathcal{N}(pq)^{1/2}\Delta. 
$$

Set $k=ap-bq$. We restrict our count to those $k$'s for which 
$$
\epsilon^{-1}|\sigma_1(k)|< |\sigma_2(k)|\le \epsilon |\sigma_1(k)|.
$$
This is acceptable since we are just interested in a non-trivial lower bound. In this case, \eqref{Ybound} is satisfied if 
\begin{equation} \label{Y0def}
|k|\le  Y_0:=\epsilon^{-1}Y=\epsilon^{-2}\mathcal{N}(pq)^{1/2}\Delta.
\end{equation}
In addition, we restrict ourselves to $k$'s such that gcd$(k,q)\approx 1$. We note that $k=ap-bq$ for some $b\in \mathcal{O}_K$ iff $k\equiv ap \bmod{q}$. 

So set
\begin{equation} \label{thedefofSq}
\mathcal{S}(q):=\left\{k\in \mathcal{O}_K : \epsilon^{-1}|\sigma_1(k)|< |\sigma_2(k)|\le \epsilon |\sigma_1(k)|, \ |k|\le Y_0,  \ \mbox{gcd}(k,q)\approx1\right\}
\end{equation}
and 
\begin{equation} \label{thedefofPq}
\mathcal{P}(q):=\{p\in \mathcal{O}_K \mbox{ prime element}:   p>0, \ p\nmid q, \ N<\mathcal{N}(p)\le 2N, \ \epsilon^{-1}|\sigma_1(p)|< |\sigma_2(p)| \le \epsilon|\sigma_1(p)|\}.
\end{equation}
We note that by Lemma \ref{sigma12}, for every prime ideal $\mathfrak{p}\nmid (q)$ satisfying $N<\mathcal{N}(\mathfrak{p})\le 2N$, there is exactly one generator $p$ of $\mathfrak{p}$ in the set $\mathcal{P}(q)$. We also note that the inequality $\epsilon^{-1}|\sigma_1(p)|< |\sigma_2(p)| \le \epsilon|\sigma_1(p)|$ implies \eqref{sigmaibound}.   
Now to prove our desired result, it suffices to establish that 
\begin{equation} \label{thatiswhatwewant}
S:=\sum\limits_{k\in \mathcal{S}(q)}\sum\limits_{\substack{p\in \mathcal{P}(q)\\ ap\equiv k \bmod{q}}} 1 >0.
\end{equation}
In fact, we shall show that $S\gg N^{2/3}$. Here we point out that this power $2/3$ is half of the power $4/3$ in the corresponding inequality \eqref{suff1'}, the reason being that the parameter $N$ in the present subsection \ref{hK1new} corresponds to $N^2$ in subsection \ref{hK1}. We could have streamlined this, but the respective choices of $N$ in subsections \ref{hK1} and \ref{hK1new} appear most natural to the authors. 

\subsubsection{Reduction to sums involving Dirichlet characters}
We detect the congruence condition above using Dirichlet characters modulo $q$, getting 
\begin{equation*}
S=\frac{1}{\varphi(q)}\cdot \sum\limits_{\chi \bmod q}  \sum\limits_{k\in \mathcal{S}(q)}\chi(ak^{-1})  \sum\limits_{p\in\mathcal{P}(q)}  \chi(p).
\end{equation*}
For every Dirichlet character modulo $q$, let $\chi_{\text{Hecke}}$ be a Hecke character belonging to $\chi$. Then recalling \eqref{Heckechar}, the above becomes
\begin{equation} \label{T ChiInfty}
S=\frac{1}{\varphi(q)}\cdot \sum\limits_{\chi \bmod{q}} \sum\limits_{k\in \mathcal{S}(q)} \chi(ak^{-1})\sum\limits_{\substack{\mathfrak{p}\in\mathbb{P}(q)\\ N<\mathcal{N}(\mathfrak{p})\le 2N}} \chi_{\text{Hecke}}(\mathfrak{p})\overline{\chi}_\infty(p),
\end{equation}
where $\mathbb{P}(q)$ is the set of non-zero prime ideals not dividing $(q)$, and $p$ is the unique generator of $\mathfrak{p}$ in $\mathcal{P}(q)$.  Recalling \eqref{wdef}, \eqref{infinitypart} and $p>0$, we have
$$
\overline{\chi}_\infty(p)=\mbox{sgn}(\sigma_2(p))^{-u_2(\chi)} \left|\frac{\sigma_1(p)}{\sigma_2(p)} \right|^{-i \pi (n(\chi) +\gamma(\chi))/\log \epsilon},
$$
where $u_2(\chi)\in \{0,1\}$, $n(\chi)\in \mathbb{Z}$ and 
$$
\gamma(\chi)=\frac{l}{2}-\frac{\arg(\chi(\epsilon))}{2\pi} 
$$
with 
$$
l:=\begin{cases} 0 & \mbox{ if } \mbox{sgn}\left(\sigma_2(\epsilon)\right)^{u_2}=1,\\
1 & \mbox{ if } \mbox{sgn}\left(\sigma_2(\epsilon)\right)^{u_2}=-1. \end{cases}
$$
We are free to choose $n(\chi)=0$ and $u_2(\chi)=0$, in which case $l=0$. Thus, the infinity part takes the form 
$$
\overline{\chi}_\infty(p)=\left| \frac{\sigma_1(p)}{\sigma_2(p)} \right|^{i\arg(\chi(\epsilon))/(2\log\epsilon)},
$$
and hence the inner sum in \eqref{T ChiInfty} becomes 
\begin{equation} \label{transformed}
\sum\limits_{\substack{\mathfrak{p}\in\mathbb{P}(q)\\ N<\mathcal{N}(\mathfrak{p})\le 2N}} \chi_{\text{Hecke}}(\mathfrak{p})\overline{\chi}_\infty(p)=\sum\limits_{\substack{\mathfrak{p}\in\mathbb{P}(q)\\ N<\mathcal{N}(\mathfrak{p})\le 2N}} \chi_{\text{Hecke}}(\mathfrak{p})\left| \frac{\sigma_1(p)}{\sigma_2(p)} \right|^{i\arg(\chi(\epsilon))/(2\log\epsilon)}.
\end{equation}

The summands are generally not Hecke characters. Our goal is to approximate them by a linear combination of Hecke characters. To this end, we will approximate the terms $\left| \frac{\sigma_1(p)}{\sigma_2(p)} \right|^{i\arg(\chi(\epsilon))/(2\log\epsilon)}$ by linear combinations of terms of the form $\psi^n(p)$, where 
$$
\psi(\alpha)=\left| \frac{\sigma_1(\alpha)}{\sigma_2(\alpha)} \right|^{i\pi /\log\epsilon}\quad \mbox{for } \alpha\in K.
$$ 
We note that $\psi(\alpha)=\psi(\alpha u)$ for any unit $u$ in $\mathcal{O}_K$, and thus 
$$
\psi_{\text{Hecke}}((\alpha)):=\psi(\alpha)
$$ 
is well-defined and a Hecke character modulo (1) on the principal ideals $(\alpha)$. In this way, we are able to reduce the right-hand side of \eqref{transformed} to sums 
$$
\sum\limits_{\substack{\mathfrak{p}\in\mathbb{P}(q)\\ N<\mathcal{N}(\mathfrak{p})\le 2N}} (\chi_{\text{Hecke}}\psi_{\text{Hecke}}^n)(\mathfrak{p})
$$
of Hecke characters over prime ideals. To estimate these sums, we will use Corollary \ref{GRHcharsum}.

\subsubsection{Reduction to linear combinations of Hecke characters}
Set
$$
\theta:=\frac{\arg(\chi(\epsilon))}{2\pi}\in \bigg(-\frac{1}{2},\frac{1}{2}\bigg]
$$
and
\begin{equation} \label{Zdef}
Z=Z(p):=\frac{\log|\sigma_1(p)/\sigma_2(p)|}{2\log \epsilon}.
\end{equation}
Using \eqref{sigmaqrel}, we have 
$$
\epsilon^{-1}\le \left| \frac{\sigma_1(p)}{\sigma_2(p)} \right|<\epsilon
$$
and hence 
$$
-1/2\le Z< 1/2.
$$
Our aim is to approximate 
\begin{equation} \label{simply}
\left| \frac{\sigma_1(p)}{\sigma_2(p)} \right|^{i\arg(\chi(\epsilon))/(2\log\epsilon)} =\exp\left(2\pi i \theta Z\right)=e(\theta Z)
\end{equation}
by a linear combination of
\begin{equation} \label{psin}
\psi^n(p)=\left| \frac{\sigma_1(p)}{\sigma_2(p)} \right|^{i\pi n/\log\epsilon}=
\exp\left(2\pi inZ\right)=e(nZ).
\end{equation}
Indeed, in $(-1/2,1/2)$, we may write $e(\theta Z)$ as a Fourier series
\begin{equation} \label{fourier}
e(\theta Z)= \sum\limits_{n\in\mathbb{Z}} a_ne(nZ),
\end{equation}
where 
\begin{equation} \label{an}
\begin{split}
a_n&= \int\limits_{-1/2}^{1/2}e((\theta-n)z)\mbox{d}z\\
&=\frac{1}{2\pi i(\theta-n)}\cdot\left(e\left(\frac{\theta-n}{2}\right)-e\left(\frac{n-\theta}{2}\right)\right)\\
&=\frac{e(-n/2)}{\pi (\theta-n)}\cdot \sin(\pi \theta),
\end{split}
\end{equation} 
taking into account that $e(n/2)=e(-n/2)$ for $n\in \mathbb{Z}$. 
The second and third lines of equation \eqref{an} above are interpreted to equal 1 if $n=0=\theta$, i.e.
$$
a_0=1 \quad \mbox{if } \theta=0.
$$

We would like to cut off the Fourier series \eqref{fourier} at an appropriate point. To this end, we establish the following approximation by a partial sum. Although the derivation of similar results is standard, we have decided to give a full proof because our coefficients $a_n$ take a specific form. 

\begin{lemma} \label{fourierapprox}
For every $W\ge 2$ and $Z\in [-1/2,1/2)$, we have
\begin{equation} \label{fzcutoff}
e(\theta Z)= \sum\limits_{|n|\le W} a_n e(nZ) + O\left(\min\left\{\log W, \frac{1}{W||Z-1/2||}\right\}\right)  
\end{equation}
with an absolute $O$-constant, 
where $||x||$ denotes the distance of $x\in \mathbb{R}$ to the nearest integer.
\end{lemma}
\begin{proof}
We first prove that the tail of the series satisfies the bound
\begin{equation} \label{tail}
\sum\limits_{|n|> W} a_ne(nZ) \ll \frac{1}{W||Z-1/2||}.  
\end{equation}
To this end, we write
\begin{equation} \label{S1S2split}
\sum\limits_{|n|> W} a_ne(nZ)=\sum\limits_{n> W} a_ne(nZ)+\sum\limits_{n<-W} a_ne(nZ)=S_1+S_2,
\end{equation}
say, and establish the bound 
\begin{equation} \label{S1bound}
S_1\ll \frac{1}{W||Z-1/2||}.
\end{equation}
The same bound for $S_2$ can be established in a similar fashion. 

Recalling \eqref{an}, we have
\begin{equation} \label{S1express}
S_1=\lim\limits_{X\rightarrow \infty} \frac{\sin(\pi \theta)}{\pi }\cdot \sum\limits_{W<n\le X}\frac{e(n(Z-1/2))}{\theta-n}.
\end{equation}
Using partial summation for any $X>W$, we transform the sum on the right-hand side into 
 \begin{equation} \label{partialsummation}
\begin{split}
 \sum\limits_{W<n\le X}\frac{e(n(Z-1/2))}{\theta-n}=& \frac{1}{\theta-X}\sum\limits_{W<n\le X} e(n(Z-1/2))-\\ &
\int_{W}^{X}\frac{\mbox{d}}{\mbox{d}t}\left(\frac{1}{\theta-t}\right)\sum\limits_{W<n\le t} e(n(Z-1/2))\mbox{d}t.
\end{split}
 \end{equation}
From a standard bound for linear exponential sums, we know that
  $$\sum\limits_{W< n\le t}e(n(Z-1/2))\ll \frac{1}{||Z-1/2||}.$$ 
Plugging this into \eqref{partialsummation}, taking $X\rightarrow \infty$ and recalling \eqref{S1express}, we obtain the desired bound \eqref{S1bound}. Bounding $S_2$ similarly and recalling \eqref{S1S2split}, we obtain the bound \eqref{tail}.   

Just using $|e(\theta Z)|=1$, $|a_0|=1$ and $|a_n|\ll 1/|n|$ for $n\not=0$, we trivially get 
\begin{equation} \label{trivially}
e(\theta Z)= \sum\limits_{|n|\le W} a_n e(n Z)+O(\log W).
\end{equation}
Now combining \eqref{fourier}, \eqref{tail} and \eqref{trivially}, we obtain the desired result.
\end{proof} 

Putting \eqref{simply}, \eqref{psin} and \eqref{fzcutoff} together, we have
$$
\left| \frac{\sigma_1(p)}{\sigma_2(p)} \right|^{i\arg(\chi(\epsilon))/(2\log\epsilon)}=
\sum\limits_{|n|\le W} a_n\psi^n(p) + O\left(\min\left\{\log W,\frac{1}{W||Z(p)-1/2||}\right\}\right),
$$
where $W\ge 2$ is a free parameter which we will later fix suitably. Noting \eqref{transformed}, we thus end up with an approximation of the form
\begin{equation} \label{finalinner}
\begin{split}
\sum\limits_{\substack{\mathfrak{p}\in\mathbb{P}(q)\\ N<\mathcal{N}(\mathfrak{p})\le  2N}} \chi_{\text{Hecke}}(\mathfrak{p})\overline{\chi}_\infty(p)=& 
\sum\limits_{|n|\le W} a_n
\sum\limits_{\substack{\mathfrak{p}\in\mathbb{P}(q)\\ N<\mathcal{N}(\mathfrak{p})\le 2N}} (\chi_{\text{Hecke}}\psi_{\text{Hecke}}^n)(\mathfrak{p})+\\ & O\left(\sum\limits_{p\in\mathcal{P}(q)} \min\left\{\log W,\frac{1}{W||Z(p)-1/2||}\right\}\right).
\end{split}
\end{equation}

\subsubsection{Estimation of the error term}
To control the $O$-term in \eqref{finalinner}, we need to obtain information about the spacing of the $Z(p)$'s. Assume from now on that $\mathcal{N}(p)>N\ge d$. Then $Z(p)=0$ iff $p$ is a rational prime, i.e., $p$ is inert in $\mathcal{O}_K$. Now we split $\mathcal{P}(q)$ into three sets
\begin{equation}
\begin{split}
\mathcal{P}_0:= & \mathcal{P}(q)\cap \mathbb{Z},\\
\mathcal{P}_+:= & \{p\in \mathcal{P}(q)\setminus \mathbb{Z} : \sigma_1(p)/\sigma_2(p)>0\},\\
\mathcal{P}_-:= & \{p\in \mathcal{P}(q)\setminus \mathbb{Z} : \sigma_1(p)/\sigma_2(p)<0\}. 
\end{split}
\end{equation}
We observe that
\begin{equation} \label{P0}
\sum\limits_{p\in\mathcal{P}_0} \min\left\{\log W,\frac{1}{W||Z(p)-1/2||}\right\}\ll \frac{N^{1/2}}{W}
\end{equation}
since $\mathcal{N}(p)=p^2$ if $p$ is a rational prime. Moreover, if 
$$
m_{\pm}:=\min\limits_{\substack{p_1,p_2\in \mathcal{P}_{\pm}\\ p_1\not=p_2}} \left| Z(p_1)-Z(p_2)\right|,
$$
then 
\begin{equation} \label{P+-}
\sum\limits_{p\in\mathcal{P}_{\pm}} \min\left\{\log W,\frac{1}{W||Z(p)-1/2||}\right\} \ll \log W+\frac{|\log m_{\pm}|}{Wm_{\pm}}.
\end{equation}

In the following, we work out a lower bound for $m_{+}$. The same lower bound can be obtained for $m_{-}$ in a similar way. Assume  that $p_1,p_2\in \mathcal{P}_+$ and $p_1\not=p_2$. Then we have 
$$
Z(p_1)-Z(p_2)= \frac{1}{2\log \epsilon}\cdot \log \frac{\sigma_1(p_1)\sigma_2(p_2)}{\sigma_2(p_1)\sigma_1(p_2)}. 
$$  
Now we use the bound 
$$
|\log x|\ge c_1|x-1|,
$$
valid for a suitable constant $c_1=c_1(\epsilon)>0$ in the interval $[\epsilon^{-2},\epsilon^2]$. 
This implies 
$$
|Z(p_1)-Z(p_2)|\ge \frac{c_1}{2\log \epsilon}\cdot \frac{|\sigma_1(p_1)\sigma_2(p_2)-\sigma_2(p_1)\sigma_1(p_2)|}{|\sigma_2(p_1)\sigma_1(p_2)|}. 
$$ 

Let $p_i=a_i+b_i\sqrt{d}$ for $i=1,2$, where $a_i,b_i\in \mathbb{Q}$. We calculate that 
$$
\sigma_1(p_1)\sigma_2(p_2)-\sigma_2(p_1)\sigma_1(p_2)=2(b_1a_2-b_2a_1)\sqrt{d}.
$$ 
If this equals 0, then it follows that 
$$
\frac{p_2}{p_1}=\frac{a_2}{a_1}
$$
and hence 
\begin{equation} \label{conc}
\frac{\mathcal{N}(p_2)}{\mathcal{N}(p_1)}=\left(\frac{a_2}{a_1}\right)^2. 
\end{equation}
Since $p_1,p_2$ are not contained in $\mathbb{Z}$ and $[K:\mathbb{Q}]=2$, both $\mathcal{N}(p_1)$ and $\mathcal{N}(p_2)$ are rational primes. These primes are distinct for otherwise, $p_1$ and $p_2$ were associates, which is not the case by construction of $\mathcal{P}(q)$. Therefore, the quotient $\mathcal{N}(p_2)/\mathcal{N}(p_1)$ is not the square of a rational number, contradicting \eqref{conc}. We conclude that 
$$
|b_1a_2-b_2a_1|\ge \frac{1}{4}
$$
and hence
$$
|Z(p_1)-Z(p_2)|\ge \frac{c_1\sqrt{d}}{4|\sigma_2(p_1)\sigma_1(p_2)|\log \epsilon}\ge 
 \frac{c_1\sqrt{d}}{4\mathcal{N}(p_1p_2)^{1/2}\epsilon\log \epsilon}
\ge \frac{c_1\sqrt{d}}{8N\epsilon \log \epsilon}.
$$
It follows that
$$
m_+\ge \frac{c_1\sqrt{d}}{8N\epsilon \log \epsilon}.
$$
Performing analogous calculations for $m_-$ and using \eqref{P0} and \eqref{P+-}, we obtain the following result.

\begin{lemma} \label{errorbound} Assume that $W\ge 2$. Then
$$
\sum\limits_{p\in\mathcal{P}(q)} \min\left\{\log W,\frac{1}{W||Z(p)-1/2||}\right\} \ll \log W+ \frac{N\log N}{W}.
$$
\end{lemma}

\subsubsection{Completion of the proof}
Combining Lemma \ref{errorbound} and \eqref{finalinner}, we get 
\begin{equation} \label{almostthere}
    \sum\limits_{\substack{\mathfrak{p}\in\mathbb{P}(q)\\ N<\mathcal{N}(\mathfrak{p})\le 2N}} \chi_{\text{Hecke}}(\mathfrak{p})\overline{\chi}_\infty(p)=\sum\limits_{|n|\le W} a_n
\sum\limits_{\substack{\mathfrak{p}\in\mathbb{P}(q)\\ N<\mathcal{N}(\mathfrak{p})\le 2N}} (\chi_{\text{Hecke}}\psi_{\text{Hecke}}^n)(\mathfrak{p})+O\left(\log W+ \frac{N\log N}{W}\right).
\end{equation}
Note that $a_n$ depends on the Dirichlet character $\chi \bmod{q}$. When $\chi$ is principal, then so is $\chi_{\text{Hecke}}$, by our choice $u_2(\chi)=0=n(\chi)$. In this case, we have $\theta=0$ and hence 
$$
a_n=\begin{cases} 1 & \mbox{ if } n=1, \\ 0 & \mbox{ otherwise.}\end{cases}
$$
Hence, using Corollary \ref{GRHcharsum}, \eqref{almostthere}, $|a_0|\le 1$ and $|a_n|\ll 1/|n|$ if $n\not=0$, we obtain
\begin{equation*}
    \sum\limits_{\substack{\mathfrak{p}\in\mathbb{P}(q)\\ N<\mathcal{N}(\mathfrak{p})\le 2N}} \chi_{\text{Hecke}}(\mathfrak{p})\overline{\chi}_\infty(p)= \delta(\chi) \int\limits_{N}^{2N}\frac{\mbox{d}t}{\log t}+O\left(N^{1/2}(\log N)(\log W)+\frac{N\log N}{W}\right)
\end{equation*}
under GRH, where 
$$
\delta(\chi)=\begin{cases} 1 & \mbox{ if } \chi \mbox{ is principal,}\\
0 & \mbox{ otherwise.} \end{cases}
$$
Choosing $W:=N^{1/2}$, it follows that  
\begin{equation*}
     \sum\limits_{\substack{\mathfrak{p}\in\mathbb{P}(q)\\ N<\mathcal{N}(\mathfrak{p})\le 2N}} \chi_{\text{Hecke}}(\mathfrak{p})\overline{\chi}_\infty(p)=\delta(\chi)\int\limits_{N}^{2N}\frac{\mbox{d}t}{\log t}+O\left(N^{1/2}\log^2 N\right).
\end{equation*}
Combining this with \eqref{T ChiInfty}, we get 
 \begin{equation} \label{Sprestep}
 \begin{split} 
S=\frac{1}{\varphi(q)}\cdot \int\limits_{N}^{2N}\frac{\mbox{d}t}{\log t}\cdot\sum\limits_{k\in \mathcal{S}(q)} 1+O\left(\frac{N^{1/2}\log^2 N}{\varphi(q)}\sum\limits_{\chi \bmod{q}}\left|\sum\limits_{k\in \mathcal{S}(q)} \chi(ak^{-1})\right|\right).
\end{split}
\end{equation}

Using the Cauchy-Schwarz inequality and orthogonality relations for Dirichlet characters, we have
\begin{equation}\label{errortermfinal}
    \begin{split}
\sum\limits_{\chi \bmod{q}}\left|\sum\limits_{k\in \mathcal{S}(q)} \chi(ak^{-1})\right| = &
    \sum\limits_{\chi \bmod q}\left|\sum\limits_{k\in \mathcal{S}(q)}\overline{\chi}(k)\right|\\ \le& \varphi(q)^{1/2}\left(\sum\limits_{\chi \bmod q}\left|\sum\limits_{k\in \mathcal{S}(q)}\overline{\chi}(k)\right|^2\right)^{1/2}\\=&\varphi(q)^{1/2}\left(\sum\limits_{\chi \bmod q}\ \sum\limits_{k_1,k_2\in \mathcal{S}(q)}\overline{\chi}(k_1)\chi(k_2)\right)^{1/2}
    \\=&\varphi(q)\left(\sum\limits_{\substack{k_1,k_2\in\mathcal{S}(q)\\ k_1\equiv k_2 \bmod q\\}}1\right)^{1/2}.
    \end{split}
\end{equation}
We claim that $k_1,k_2\in \mathcal{S}(q)$ and $k_1\equiv k_2 \bmod q$ imply $k_1=k_2$.
To see this, assume that $k_1,k_2\in \mathcal{S}(q)$, $k_1\equiv k_2 \bmod{q}$ and $k_1\neq k_2$. Since $q|(k_1-k_2)$, $\mathcal{N}(q)|\mathcal{N}(k_1-k_2)$ and so $\mathcal{N}(q)\le \mathcal{N}(k_1-k_2).$ Let $k_i=a_i+b_i\sqrt{d}$ for $i=1,2$. Then $\mathcal{N}(k_1-k_2)=(a_1-a_2)^2-(b_1-b_2)^2d.$ 
If $k_1,k_2\in \mathcal{S}(q)$, then $a_i,b_i\ll\mathcal{N}(k_i)^{1/2}$ for $i=1,2$, where we recall that $|\sigma_1(k_i)\sigma_2(k_i)|=\mathcal{N}(k_i)$. Hence $\mathcal{N}(k_1-k_2)\ll \mathcal{N}(k_1)+\mathcal{N}(k_2)\ll Y_0^2.$
However, 
\begin{equation} \label{Y0ref}
Y_0^2\asymp \frac{\mathcal{N}(p)}{\mathcal{N}(q)}\asymp \frac{N}{\mathcal{N}(q)}
\end{equation}  
in view of \eqref{Deltadef} and \eqref{Y0def}, and 
$$
N = o\left(\mathcal{N}(q)^2\right)
$$ 
by \eqref{Ndef} and our assumption $\tau<2$. So altogether, we obtain $\mathcal{N}(q)\le \mathcal{N}(k_1-k_2)= o\left(\mathcal{N}(q)\right)$ as $\mathcal{N}(q)\rightarrow \infty$. This gives a contradiction if $\mathcal{N}(q)$ is sufficiently large. Hence $k_1=k_2$. 
    
The above claim implies that 
\begin{equation} \label{impli}
\sum\limits_{\substack{k_1\equiv k_2 \bmod q\\k_1,k_2\in\mathcal{S}(q)}}1=\sum\limits_{k\in \mathcal{S}(q)} 1.
\end{equation}
Again using $|\sigma_1(k)\sigma_2(k)|=\mathcal{N}(k)$, there exists a constant $c_2>0$ such that $\mathcal{N}(k)\le c_2Y_0^2$ for every $k\in \mathcal{S}(q)$. Moreover, by Lemma \ref{sigma12} and construction of $\mathcal{S}(q)$, for every integral ideal $\mathfrak{a}$ with $\mathcal{N}(\mathfrak{a})\le c_2Y_0^2$, there exist at most two $k\in \mathcal{S}(q)$ such that $(k)=\mathfrak{a}$. It follows that 
\begin{equation}\label{s(c)1}
    \sum\limits_{k\in\mathcal{S}(q)}1 \le 2\sum\limits_{\substack{\mathfrak{a} \text{ integral ideal}\\ \mathcal{N}(\mathfrak{a})\le c_2Y_0^2 }} 1.
\end{equation}
Furthermore, Theorem \ref{Landau} implies that 
\begin{equation} \label{s(c)2}
\sum\limits_{\substack{\mathfrak{a} \text{ integral ideal}\\ \mathcal{N}(\mathfrak{a})\le c_2Y_0^2 }} 1 = A_Kc_2Y_0^2+O\left(Y_0^{2/3}\right)
\end{equation}
for some constant $A_K$ depending only on the number field $K$. Combining \eqref{errortermfinal}, \eqref{impli}, \eqref{s(c)1} and \eqref{s(c)2}, we obtain
\begin{equation}\label{errorfinalfinal}
    \sum\limits_{\chi \bmod q}\left|\sum\limits_{k\in \mathcal{S}(q)}\chi(ak^{-1})\right| =O\left(\varphi(q)Y_0\right).
\end{equation} 

Above, we derived an upper bound for the sum
$$
\sum\limits_{k\in \mathcal{S}(q)} 1.
$$
We also need a lower bound since this sum appears in the main term on the right-hand side of \eqref{Sprestep}. Again using $|\sigma_1(k)\sigma_2(k)|=\mathcal{N}(k)$, we observe that $|k|\le Y_0$ holds if $\epsilon^{-1}|\sigma_1(k)|<|\sigma_2(k)|\le \epsilon|\sigma_1(k)|$ and $\mathcal{N}(k)\le c_3Y_0^2$ for a suitable constant $c_3>0$. Moreover, by Lemma \ref{sigma12}, for every integral ideal $\mathfrak{a}$, there exists at least one generator $k$ of $\mathfrak{a}$ satisfying $\epsilon^{-1}|\sigma_1(k)|<|\sigma_2(k)|\le \epsilon|\sigma_1(k)|$. It follows that 
\begin{equation} \label{lowerbound}
        \sum\limits_{k\in\mathcal{S}(q)}1\ge  \sum\limits_{\substack{\mathcal{N}(\mathfrak{a})\le c_3Y_0^2 \\(\mathfrak{a},(q))=(1)}}1.
\end{equation}
Detecting the coprimality condition above using the M\"obius function and using Theorem \ref{Landau} and \eqref{relation1} and \eqref{relation2}, we deduce that
\begin{equation} \label{Tfinal}
\begin{split}
\sum\limits_{\substack{\mathcal{N}(\mathfrak{a})\le c_3Y_0^2 \\(\mathfrak{a},(q))=(1)}}1 =&\sum\limits_{\mathfrak{d}|(q)}\mu(\mathfrak{d})\sum\limits_{\substack{\mathcal{N}(\mathfrak{a})\le c_3Y_0^2 \\ \mathfrak{d}|\mathfrak{a}}}1\\=&\sum\limits_{\mathfrak{d}|(q)}\mu(\mathfrak{d})\sum\limits_{\substack{\mathcal{N}(\mathfrak{b})\le c_3Y_0^2/\mathcal{N}(\mathfrak{d})}}1\\=&\sum\limits_{\mathfrak{d}|(q)}\mu(\mathfrak{d})\left(A_Kc_3\cdot \frac{Y_0^2}{\mathcal{N}(\mathfrak{d})}+O\left(Y_0^{2/3}\mathcal{N}(\mathfrak{d})^{-1/3}\right)\right)\\
=& A_Kc_3 Y_0^2\cdot\frac{\varphi(q)}{\mathcal{N}(q)}+O\left(Y_0^{2/3} \mathcal{N}(q)^{\varepsilon}\right).
\end{split}
\end{equation}

Now combining \eqref{Sprestep}, \eqref{errorfinalfinal} and \eqref{Tfinal}, and using \eqref{relation2} again, we get
\begin{equation*} 
      S\ge \frac{A_Kc_3Y_0^2}{\mathcal{N}(q)}\cdot \int\limits_{N}^{2N}\frac{\mbox{d}t}{\log t}+O\left(N Y_0^{2/3}\mathcal{N}(q)^{-1+2\varepsilon}+Y_0N^{1/2+\varepsilon}\right).
\end{equation*}
Using \eqref{Y0ref}, this implies
\begin{equation} \label{finish} 
      S\ge \frac{A_Kc_4N}{\mathcal{N}(q)^2}\cdot \int\limits_{N}^{2N}\frac{\mbox{d}t}{\log t}+O\left(N^{4/3}\mathcal{N}(q)^{-4/3+2\varepsilon}+N^{1+\varepsilon}\mathcal{N}(q)^{-1/2}\right)
\end{equation}
for a suitable constant $c_4>0$. We recall our choice $\tau:=1/(2/3-\varepsilon)$ in \eqref{choice}, which implies $N=\mathcal{N}(q)^{1/(2/3-\varepsilon)}$ by \eqref{Ndef}. Equivalently, $\mathcal{N}(q)=N^{2/3-\varepsilon}$. With this choice, we obtain \eqref{thatiswhatwewant} from \eqref{finish}, provided that $\varepsilon$ is small enough and $N$ is sufficiently large. This completes the proof of Theorem \ref{secondmainresult} in the case when $h_K=1$. 

\subsection{The case of an arbitrary class number $h_K$} The modifications needed when the class number is not equal to 1 are very similar to those in subsection \ref{arb}. Therefore, we confine ourselves to a brief description and omit a number of calculations in this concluding part. First, we need a sharpened version of the Dirichlet approximation theorem for real quadratic fields, similar to that for imaginary quadratic fields in Theorem \ref{Dirichlet3}.

\begin{theorem} \label{Dirichlet4}  
Let $K$ be a real quadratic number field. Then there exists a positive constant $C$ depending at most on $K$ such that for every pair $(x_1, x_2) \in \mathbb{R}^2 \setminus \sigma(K)$, there exists an infinite sequence of pairs  $(a,q)\in\mathcal{O}_K\times(\mathcal{O}_K\setminus\{0 \})$ such that $\mathcal{N}(\mbox{gcd}((a),(q))^{-1}(q))$ increases to infinity and   
\begin{equation} \label{Dirichl} 
\left|x_i-\frac{\sigma_i(a)}{\sigma_i(q)}\right| \le \frac{C}{\mathcal{N}(q)} \quad \mbox{ for } i=1,2. 
\end{equation} 
\end{theorem} 
\begin{proof} 
Using Theorem \ref{Dirichlet1}, there exists a sequence of pairs $(a_n,q_n)\in\mathcal{O}_K\times(\mathcal{O}_K\setminus\{0 \})$ satisfying the inequalities \eqref{Dirichl} and $\mathcal{N}(q_n)<\mathcal{N}(q_{n+1})$ for all $n\in \mathbb{N}.$ Plugging the pairs $(a,q)=(a_n,q_n), (a_{n+1},q_{n+1})$ into \eqref{Dirichl} and applying the triangle inequality, we deduce that 
\begin{equation*}
\left|\frac{\sigma_i(a_n)}{\sigma_i(q_n)}-\frac{\sigma_i(a_{n+1})}{\sigma_i(q_{n+1})}\right| < \frac{2C}{\mathcal{N}(q_n)} \quad \mbox{ for } i=1,2,
\end{equation*} 
and hence 
\begin{equation*}
\left|\sigma_i(a_nq_{n+1}-a_{n+1}q_n)\right|<2C\cdot \frac{|\sigma_i(q_n)\sigma_i(q_{n+1})|}{\mathcal{N}(q_n)}  \quad \mbox{ for } i=1,2.
\end{equation*}
Multiplying the above inequality for $i=1$ with the same inequality for $i=2$, we obtain 
$$
\mathcal{N}(a_nq_{n+1}-a_{n+1}q_n)<(2C)^2 \cdot \frac{\mathcal{N}(q_{n+1})}{\mathcal{N}(q_n)}.
$$
Since 
\begin{equation*}
\mathcal{N}\left(\gcd((a_{n+1}),(q_{n+1}))\right)\le\mathcal{N}(a_nq_{n+1}-a_{n+1}q_n),
\end{equation*}
it follows that 
\begin{equation*}
\mathcal{N}\left(\gcd((a_{n+1}),(q_{n+1}))^{-1}(q_{n+1})\right)>(2C)^{-2}\mathcal{N}(q_n).
\end{equation*}
This shows that the sequence of $\mathcal{N}\left(\gcd((a_{n+1}),(q_{n+1}))^{-1}(q_{n+1})\right)$ diverges to infinity, which completes the proof.
\end{proof}   
Our analysis for the case of an arbitrary class number $h_K$ begins as in the previous subsection, but led by similar considerations as in subsection \ref{arb}, we replace the definition of $\Delta$ in \eqref{Deltadef} by
$$
\Delta:=\frac{C}{\mathcal{N}(\mathfrak{q})}
$$
(which we are free to do since $C/\mathcal{N}(\mathfrak{q})\ge C/\mathcal{N}(q)$),  
with the effect that now 
$$
Y_0^2\asymp \frac{N \mathcal{N}(q)}{\mathcal{N}(\mathfrak{q})^2}
$$
in place of \eqref{Y0ref}. Moreover, similarly as in subsection \ref{arb}, it is required to replace the condition $p\nmid q$ by $(p)\nmid \mathfrak{q}$ and the coprimality condition $(k,q)\approx 1$ by $\mathfrak{C}=\mathfrak{D}$, where we set 
$$
\mathfrak{C}:=\mbox{gcd}((k),(q)), \quad \mathfrak{D}:=\mbox{gcd}((a),(q))\quad \mbox{and} \quad \mathfrak{q}:=\mathfrak{D}^{-1}(q). 
$$
Consequently, the sets $\mathcal{S}(q)$ and $\mathcal{P}(q)$, previously defined in \eqref{thedefofSq} and \eqref{thedefofPq}, will be modified into
\begin{equation*}
\mathcal{S}(q):=\left\{k\in \mathcal{O}_K : \epsilon^{-1}|\sigma_1(k)|< |\sigma_2(k)|\le \epsilon |\sigma_1(k)|, \ |k|\le Y_0,  \ \mathfrak{C}=\mathfrak{D}\right\}
\end{equation*}
and 
\begin{equation*} 
\mathcal{P}(q):=\{p\in \mathcal{O}_K \mbox{ prime element}:   p>0, \ (p)\nmid \mathfrak{q}, \ N<\mathcal{N}(p)\le 2N, \ \epsilon^{-1}|\sigma_1(p)|< |\sigma_2(p)| \le \epsilon|\sigma_1(p)|\}.
\end{equation*}
Further, in the sums over $\mathfrak{p}$ in \eqref{T ChiInfty} and the following expressions, we replace the set $\mathbb{P}(q)$ by the set $\mathbb{P}_0(q)$ of 
all non-zero {\it principal} prime ideals not dividing $\mathfrak{q}$. The transformations of the said sums over $\mathfrak{p}$ are as in subsection \ref{hK1new}, but after reaching \eqref{almostthere} and before applying Corollary \ref{GRHcharsum}, we detect the principality of $\mathfrak{p}$ using class group characters $\psi$. As in subsection \ref{arb}, this creates harmless additional factors of $\psi(\mathfrak{p})$ which don't change the subsequent calculations except that an additional factor of $1/h_K$ is being introduced in the main terms.  In place of \eqref{Sprestep}, we derive the asymptotic formula
 \begin{equation*}
      S=  \frac{1}{h_k\varphi(\mathfrak{q})} \cdot \int\limits_{N}^{2N} \frac{\mbox{d}t}{\log t} \cdot  \sum\limits_{k\in \mathcal{S}(q)} 1 +
O\left(\frac{N\log^2 N }{\varphi(\mathfrak{q})}\cdot \sum\limits_{\chi \bmod{\mathfrak{q}}}
\left| \sum\limits_{k\in \mathcal{S}(q)} \chi(ak^{-1}) \right|\right),
  \end{equation*}
which should be compared to \eqref{SS2} in subsection \ref{arb}. Here we define $S$ as in \eqref{thatiswhatwewant}. To bound the $O$-term from above and the main term from below, generalizing the corresponding calculations in the previous subsection, we now need to incorporate a few extra arguments taking care of the condition $\mathfrak{C}=\mathfrak{D}$ in our redefined $S(q)$ and the principality of the ideal $(k)$, similarly as we did in subsection \ref{arb}. Eventually, we arrive at the same lower bound as in \eqref{finish}, with the term $\mathcal{N}(q)$ replaced by $\mathcal{N}(\mathfrak{q})$ at all relevant places. Here we note that although the parameter $Y_0$ depends on $\mathcal{N}(q)$, the final estimate depends only on $\mathcal{N}(\mathfrak{q})$. This is a result of the condition $\mathfrak{C}=\mathfrak{D}$ in place of $(k,q)\approx 1$, which lowers the sum $\sum\limits_{k\in \mathcal{S}(q)} 1$ essentially by a factor of $1/\mathcal{N}(\mathfrak{D})$. The final choice of $N$ will now be $N=\mathcal{N}(\mathfrak{q})^{1/(2/3-\varepsilon)}$, and as in subsection \ref{hK1new}, we obtain a non-trivial lower bound for the sum $S$. This completes the proof in the case of an arbitrary class number $h_K$.

\end{document}